\documentclass[12pt]{article}

\usepackage{amsfonts,amssymb,amsmath,amsthm,epsfig,euscript}

\usepackage{tikz}
\usetikzlibrary{matrix,trees,arrows}

\usetikzlibrary{positioning}
\usetikzlibrary{fit}
\usetikzlibrary{patterns}

\newtheorem{theorem}{Theorem}
\newtheorem{lemma}[theorem]{Lemma}

\theoremstyle{definition}
{

\newtheorem{problem}{Problem}
}

\long\def\symbolfootnote[#1]#2{\begingroup
\def\thefootnote{\fnsymbol{footnote}}\footnote[#1]{#2}\endgroup}

\newcommand{\red}[1][\sigma]{\mathrm{red}(#1)}

\newcommand{\sg}{\sigma}

\newcommand{\mmp}{\mathrm{mmp}}
\newcommand{\MMP}{\mathrm{MMP}}

\def\A{\mathcal{A}}

\newcommand{\fig}[2]{\begin{figure}[ht]
\centerline{\scalebox{.66}{\epsfig{file=#1.eps}}}
\caption{#2}
\label{fig:#1}
\end{figure}}

\newcommand{\shadetheboxes}[1]{
	\foreach \x/\y in {#1}
      	\fill[pattern color = black!65, pattern=north east lines] (\x,\y) rectangle +(1,1);
	}
	
\newcommand{\drawthegrid}[1]{
	\draw (0.01,0.01) grid (#1+0.99,#1+0.99);
	}
	
\newcommand{\drawtheclpattern}[1]{
	\foreach \x/\y in {#1}
      	\filldraw (\x,\y) circle (6pt);
	}

\newcommand{\drawspecialbox}[1]{
	\foreach \x/\y/\z/\w/\A in {#1}
		{
       		\fill[color = white!100, opacity=1, rounded corners = 1.5pt] (\x+0.125,\y+0.125) rectangle (\z-0.125,\w-0.125);
       		\draw[color = black, rounded corners = 1.5pt] (\x+0.125,\y+0.125) rectangle (\z-0.125,\w-0.125);
       		\fill[black] (\x/2+\z/2,\y/2+\w/2) node {$\scriptstyle\A$};
       	}
    }

\newcommand{\mmpattern}[5]{									
  \raisebox{0.6ex}{
  \begin{tikzpicture}[scale=0.35, baseline=(current bounding box.center), #1]
  \useasboundingbox (0.0,-0.1) rectangle (#2+1.4,#2+1.1);
    
    \shadetheboxes{#4}
    
    \drawthegrid{#2}
    
    \drawspecialbox{#5}
    
    \drawtheclpattern{#3}

  \end{tikzpicture}}
}

\title{Quadrant marked mesh patterns in $132$-avoiding permutations III}

\author{
Sergey Kitaev \\
\small University of Strathclyde\\[-0.8ex]
\small Livingstone Tower, 26 Richmond Street\\[-0.8ex]
\small Glasgow G1 1XH, United Kingdom\\[-0.8ex]
\small \texttt{sergey.kitaev@cis.strath.ac.uk}
\and
Jeffrey Remmel \\
\small Department of Mathematics\\[-0.8ex]
\small University of California, San Diego\\[-0.8ex]
\small La Jolla, CA 92093-0112. USA\\[-0.8ex]
\small \texttt{jremmel@ucsd.edu}
\and
Mark Tiefenbruck\\[-0.8ex]
\small Department of Mathematics\\[-0.8ex]
\small University of California, San Diego\\[-0.8ex]
\small La Jolla, CA 92093-0112. USA\\[-0.8ex]
\small \texttt{mtiefenb@math.ucsd.edu}
}

\date{\small Submitted: Date 1;  Accepted: Date 2;
 Published: Date 3.\\
\small MR Subject Classifications: 05A15, 05E05}

\begin{document}
\maketitle

\begin{abstract}
\noindent 
Given a permutation $\sg = \sg_1 \ldots \sg_n$ in the symmetric group 
$S_n$, we say that $\sg_i$ matches the marked mesh pattern 
$MMP(a,b,c,d)$ in $\sg$ if there are at least 
$a$ points to the right of $\sg_i$ in $\sg$ which are greater than 
$\sg_i$, at least $b$ points to the left of $\sg_i$ in $\sg$ which 
are greater than $\sg_i$,  at least $c$ points to the left of 
$\sg_i$ in $\sg$ which are smaller  than $\sg_i$, and 
at least  $d$ points to the right of $\sg_i$ in $\sg$ which 
are smaller than $\sg_i$.

This paper is continuation of the systematic study of the distribution 
of quadrant marked mesh patterns in 132-avoiding permutations 
started in \cite{kitremtie} and \cite{kitremtieII} where 
we studied the distribution of the number of matches 
of $MMP(a,b,c,d)$ in 132-avoiding permutations 
where at most two elements of of  $a,b,c,d$ are greater 
than zero and the remaining elements are zero. In this paper, 
we study  the distribution of the number of matches 
of $MMP(a,b,c,d)$ in 132-avoiding permutations 
where at least three of $a,b,c,d$ are greater 
than zero. 
We provide explicit recurrence relations to enumerate our objects which 
can be used to give closed forms for the generating functions associated 
with such distributions. In many  cases, we provide combinatorial explanations of the coefficients that appear in our generating functions. \\

\noindent {\bf Keywords:} permutation statistics, quadrant marked mesh pattern, distribution
\end{abstract}

\tableofcontents

\section{Introduction}

The notion of mesh patterns was introduced by Br\"and\'en and Claesson \cite{BrCl} to provide explicit expansions for certain permutation statistics as, possibly infinite, linear combinations of (classical) permutation patterns.  This notion was further studied in \cite{AKV,HilJonSigVid,kitlie,kitrem,kitremtie,Ulf}.

Kitaev and Remmel \cite{kitrem} initiated the systematic study of distribution of quadrant marked mesh patterns on permutations. The study was extended to 132-avoiding permutations by Kitaev, Remmel and Tiefenbruck in \cite{kitremtie,kitremtieII}, and the present paper continues this line of research. 
Kitaev and Remmel also studied the distribution of quadrant marked 
mesh patterns 
in up-down and down-up permutations \cite{kitrem2,kitrem3}. 

Let $\sigma = \sg_1 \ldots \sg_n$ be a permutation written in one-line notation. Then we will consider the 
graph of $\sg$, $G(\sg)$, to be the set of points $(i,\sg_i)$ for 
$i =1, \ldots, n$.  For example, the graph of the permutation 
$\sg = 471569283$ is pictured in Figure 
\ref{fig:basic}.  Then if we draw a coordinate system centered at a 
point $(i,\sg_i)$, we will be interested in  the points that 
lie in the four quadrants I, II, III, and IV of that 
coordinate system as pictured 
in Figure \ref{fig:basic}.  For any $a,b,c,d \in  
\mathbb{N} = \{0,1,2, \ldots \}$ and any $\sg = \sg_1 \ldots \sg_n \in S_n$, the set of all permutations of length $n$, we say that $\sg_i$ matches the 
quadrant marked mesh pattern $\MMP(a,b,c,d)$ in $\sg$ if, in $G(\sg)$  
relative 
to the coordinate system which has the point $(i,\sg_i)$ as its  
origin, there are at least $a$ points in quadrant I, 
at least $b$ points in quadrant II, at least $c$ points in quadrant 
III, and at least $d$ points in quadrant IV.  
For example, 
if $\sg = 471569283$, the point $\sg_4 =5$  matches 
the marked mesh pattern $\MMP(2,1,2,1)$ since, in $G(\sg)$ relative 
to the coordinate system with the origin at $(4,5)$,  
there are 3 points in quadrant I, 
1 point in quadrant II, 2 points in quadrant III, and 2 points in 
quadrant IV.  Note that if a coordinate 
in $\MMP(a,b,c,d)$ is 0, then there is no condition imposed 
on the points in the corresponding quadrant. 

In addition, we considered patterns  $\MMP(a,b,c,d)$ where 
$a,b,c,d \in \mathbb{N} \cup \{\emptyset\}$. Here when 
a coordinate of $\MMP(a,b,c,d)$ is the empty set, then for $\sg_i$ to match  
$\MMP(a,b,c,d)$ in $\sg = \sg_1 \ldots \sg_n \in S_n$, 
it must be the case that there are no points in $G(\sg)$ relative 
to the coordinate system with the origin at $(i,\sg_i)$ in the corresponding 
quadrant. For example, if $\sg = 471569283$, the point 
$\sg_3 =1$ matches 
the marked mesh pattern $\MMP(4,2,\emptyset,\emptyset)$ since  
in $G(\sg)$ relative 
to the coordinate system with the origin at $(3,1)$, 
there are 6 points in quadrant I, 
2 points in quadrant II, no  points in quadrants III  and IV.   We let 
$\mmp^{(a,b,c,d)}(\sg)$ denote the number of $i$ such that 
$\sg_i$ matches $\MMP(a,b,c,d)$ in $\sg$.

\fig{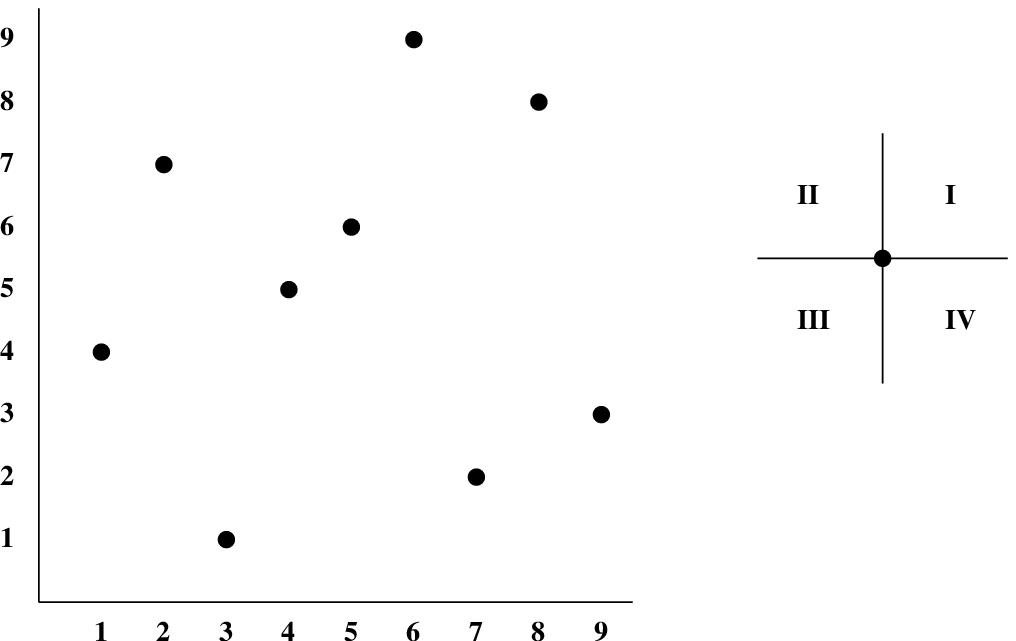}{The graph of $\sg = 471569283$.}

Note how the (two-dimensional) notation of \'Ulfarsson \cite{Ulf} for {\em marked mesh patterns} corresponds to our (one-line) notation for quadrant marked mesh patterns. For example,

\[
\MMP(0,0,k,0)=\mmpattern{scale=2.3}{1}{1/1}{}{0/0/1/1/k}\hspace{-0.25cm},\  \MMP(k,0,0,0)=\mmpattern{scale=2.3}{1}{1/1}{}{1/1/2/2/k}\hspace{-0.25cm},
\]

\[
\MMP(0,a,b,c)=\mmpattern{scale=2.3}{1}{1/1}{}{0/1/1/2/a} \hspace{-2.07cm} \mmpattern{scale=2.3}{1}{1/1}{}{0/0/1/1/b} \hspace{-2.07cm} \mmpattern{scale=2.3}{1}{1/1}{}{1/0/2/1/c} \ \mbox{ and }\ \ \ \MMP(0,0,\emptyset,k)=\mmpattern{scale=2.3}{1}{1/1}{0/0}{1/0/2/1/k}\hspace{-0.25cm}.
\]

Given a sequence $w = w_1 \ldots w_n$ of distinct integers,
let $\red[w]$ be the permutation found by replacing the
$i$-th largest integer that appears in $\sg$ by $i$.  For
example, if $\sg = 2754$, then $\red[\sg] = 1432$.  Given a
permutation $\tau=\tau_1 \ldots \tau_j$ in the symmetric group $S_j$, we say that the pattern $\tau$ {\em occurs} in $\sg = \sg_1 \ldots \sg_n \in S_n$ provided   there exists 
$1 \leq i_1 < \cdots < i_j \leq n$ such that 
$\red[\sg_{i_1} \ldots \sg_{i_j}] = \tau$.   We say 
that a permutation $\sg$ {\em avoids} the pattern $\tau$ if $\tau$ does not 
occur in $\sg$. Let $S_n(\tau)$ denote the set of permutations in $S_n$ 
which avoid $\tau$. In the theory of permutation patterns,   $\tau$ is called a {\em classical pattern}. See \cite{kit} for a comprehensive introduction to 
patterns in permutations. 

It has been a rather popular direction of research in the literature on permutation patterns to study permutations avoiding a 3-letter pattern subject to extra restrictions (see \cite[Subsection 6.1.5]{kit}). In \cite{kitremtie},
we started the study of the generating functions 
\begin{equation*} \label{Rabcd}
Q_{132}^{(a,b,c,d)}(t,x) = 1 + \sum_{n\geq 1} t^n  Q_{n,132}^{(a,b,c,d)}(x)
\end{equation*}
where for  any $a,b,c,d \in \{\emptyset\} \cup \mathbb{N}$, 
\begin{equation*} \label{Rabcdn}
Q_{n,132}^{(a,b,c,d)}(x) = \sum_{\sg \in S_n(132)} x^{\mmp^{(a,b,c,d)}(\sg)}.
\end{equation*}
For any $a,b,c,d$, we will write $Q_{n,132}^{(a,b,c,d)}(x)|_{x^k}$ for 
the coefficient of $x^k$ in $Q_{n,132}^{(a,b,c,d)}(x)$.

There is one obvious symmetry for such generating functions which is induced 
by the fact that if $\sg \in S_n(132)$, then $\sg^{-1} \in S_n(132)$. 
That is, the following lemma was proved in  \cite{kitremtie}.

\begin{lemma}\label{sym} {\rm (\cite{kitremtie})}
For any $a,b,c,d \in \{\emptyset\} \cup \mathbb{N}$, 
\begin{equation*}
Q_{n,132}^{(a,b,c,d)}(x) = Q_{n,132}^{(a,d,c,b)}(x). 
\end{equation*}
\end{lemma}

In \cite{kitremtie}, we studied the generating 
functions $Q_{132}^{(k,0,0,0)}(t,x)$,  
$Q_{132}^{(0,k,0,0)}(t,x) = Q_{132}^{(0,0,0,k)}(t,x)$, and 
$Q_{132}^{(0,0,k,0)}(t,x)$  where $k$ can be either 
the empty set or a positive integer as well as the 
generating functions $Q_{132}^{(k,0,\emptyset,0)}(t,x)$ and 
$Q_{132}^{(\emptyset,0,k,0)}(t,x)$. In  \cite{kitremtieII}, we studied the generating functions $Q_{n,132}^{(k,0,\ell,0)}(t, x)$, 
$Q_{n,132}^{(k,0,0,\ell)}(t,x)=Q_{n,132}^{(k,\ell,0,0)}(t,x)$, 
$Q_{n,132}^{(0,k,\ell,0)}(t,x)=Q_{n,132}^{(0,0,\ell,k)}(t,x)$, and 
$Q_{n,132}^{(0,k,0,\ell)}(t,x)$, where $k,\ell \geq 1$. We also showed 
that sequences of the form $(Q_{n,132}^{(a,b,c,d)}(x)|_{x^r})_{n \geq s}$ 
count a variety of combinatorial objects that appear 
in the {\em On-line Encyclopedia of Integer Sequences} (OEIS) \cite{oeis}.
Thus, our results gave new combinatorial interpretations 
of certain classical sequences such as the Fine numbers and the Fibonacci 
numbers  as well as provided certain sequences that appear in the OEIS 
with a combinatorial interpretation where none had existed before. Another particular result of our studies in \cite{kitremtie} is enumeration of permutations avoiding simultaneously the patterns 132 and 1234, while in \cite{kitremtieII},
 we made a link to the {\em Pell numbers}.

The main goal of this paper is to continue the study of 
 $Q_{132}^{(a,b,c,d)}(t,x)$ 
and combinatorial interpretations of sequences of the form 
$(Q_{n,132}^{(a,b,c,d)}(x)|_{x^r})_{n \geq s}$ 
in the case where $a,b,c,d \in \mathbb{N}$ and at least three of 
these parameters are non-zero.

Next we list the key results from 
 \cite{kitremtie} and \cite{kitremtieII}  which we need in this paper.  

\begin{theorem}\label{thm:Qk000} (\cite[Theorem 4]{kitremtie})
\begin{equation*}\label{eq:Q0000}
Q_{132}^{(0,0,0,0)}(t,x) =  C(xt) = \frac{1-\sqrt{1-4xt}}{2xt}
\end{equation*}
and, for $k \geq 1$, 
\begin{equation*}\label{Qk000}
Q_{132}^{(k,0,0,0)}(t,x) = \frac{1}{1-tQ_{132}^{(k-1,0,0,0)}(t,x)}.
\end{equation*}
Hence 
\begin{equation*}\label{eq:Q100(0)}
Q_{132}^{(1,0,0,0)}(t,0) = \frac{1}{1-t}
\end{equation*}
and, for $k \geq 2$, 
\begin{equation*}\label{x=0Qk000}
Q_{132}^{(k,0,0,0)}(t,0) = \frac{1}{1-tQ_{132}^{(k-1,0,0,0)}(t,0)}.
\end{equation*}
\end{theorem}

\begin{theorem}\label{thm:Q00k0} (\cite[Theorem 8]{kitremtie}) 
For $k \geq 1$, 
\begin{align}\label{gf00k0}
Q_{132}^{(0,0,k,0)}(t,x)&=\frac{1+(tx-t)(\sum_{j=0}^{k-1}C_jt^j) - 
\sqrt{(1+(tx-t)(\sum_{j=0}^{k-1}C_jt^j))^2 -4tx}}{2tx}\nonumber\\
&=\frac{2}{1+(tx-t)(\sum_{j=0}^{k-1}C_jt^j) + \sqrt{(1+(tx-t)(\sum_{j=0}^{k-1}C_jt^j))^2 -4tx}}\notag
\end{align}
and  
\begin{equation*}
Q_{132}^{(0,0,k,0)}(t,0) = \frac{1}{1-t(C_0+C_1 t+\cdots +C_{k-1}t^{k-1})}.
\end{equation*}
\end{theorem}

\begin{theorem}\label{thm:Qk0l0}(\cite[Theorem 5]{kitremtieII})  For all $k, \ell \geq 1$, 
\begin{equation}\label{k0l0gf}
Q_{132}^{(k,0,\ell,0)}(t,x) = 
\frac{1}{1-t Q_{132}^{(k-1,0,\ell,0)}(t,x)}.
\end{equation}
\end{theorem}

\begin{theorem} \label{thm:k00l}(\cite[Theorem 11]{kitremtieII})
For all $k, \ell \geq 1$, 
\begin{multline}\label{Qk00lgf-}
Q_{132}^{(k,0,0,\ell)}(t,x) = \\
\frac{C_\ell t^\ell + \sum_{j=0}^{\ell -1} C_j t^j (1 -tQ_{132}^{(k-1,0,0,0)}(t,x)
+t(Q_{132}^{(k-1,0,0,\ell-j)}(t,x)-\sum_{s=0}^{\ell -j -1}C_s t^s))}{1-tQ_{132}^{(k-1,0,0,0)}(t,x)}.
\end{multline}
\end{theorem}

\begin{theorem}\label{thm:Q0kl0}(\cite[Theorem 14]{kitremtieII})
 For all $k,\ell \geq 1$, 
\begin{multline}\label{Q0kl0gf}
Q_{132}^{(0,k,\ell,0)}(t,x) = \\
\frac{C_{k-1} t^{k-1} + \sum_{j=0}^{k-2} C_j t^j \left(1 -tQ_{132}^{(0,0,\ell,0)}(t,x) 
+t(Q_{132}^{(0,k-i-1,\ell,0)}(t,x)-\sum_{s=0}^{k-i-2}C_s t^s)\right)}{1-tQ_{132}^{(0,0,\ell,0)}(t,x)}.
\end{multline}
\end{theorem}

\begin{theorem}\label{thm:Q0k0l}(\cite[Theorem 17]{kitremtieII}) For all $k,\ell \geq 1$, 
\begin{equation}\label{Q0k0lgf2}
Q_{132}^{(0,k,0,\ell)}(t,x) = \frac{\Phi_{k,\ell}(t,x)}{1-t}
\end{equation}
where \\
$\displaystyle 
\Phi_{k,\ell}(t,x) =  \sum_{j=0}^{k+\ell -1}C_jt^j -\sum_{j=0}^{k+\ell-2}C_jt^{j+1}
+t\left(\sum_{j=0}^{k-2} C_j t^j \left(Q_{132}^{(0,k,0,\ell-j-1)}(t,x) - 
\sum_{s=0}^{k-j-2} C_s t^s\right)\right) + $ \\
$\displaystyle  t \left(Q_{132}^{(0,k,0,0)}(t,x) - 
\sum_{u=0}^{k-1} C_u t^u\right)\left(Q_{132}^{(0,0,0,\ell)}(t,x) - 
\sum_{v=0}^{\ell-1} C_v t^v\right)+ $ \\
$\displaystyle t \left(\sum_{j=1}^{\ell-1} C_j t^j \left(Q_{132}^{(0,k,0,\ell-j)}(t,x) - 
\sum_{w=0}^{k+\ell-j-2} C_w t^w\right)\right)$.
\end{theorem}

As it was pointed out in \cite{kitremtie}, {\em avoidance} of a marked mesh pattern without quadrants containing the empty set can always be expressed in terms of multi-avoidance of (possibly many) classical patterns. Thus, among our results we will re-derive several known facts in permutation patterns theory. However, our main goals are more ambitious aimed at finding  distributions in question.

\section{$Q_{n,132}^{(k,0,m,\ell)}(x)=Q_{n,132}^{(k,\ell,m,0)}(x)$
 where $k,\ell,m \geq 1$}

By Lemma \ref{sym}, we know that 
$Q_{n,132}^{(k,0,m,\ell)}(x)= Q_{n,132}^{(k,\ell,m,0)}(x)$. 
Thus, we will only consider $Q_{n,132}^{(k,\ell,m,0)}(x)$ in 
this section.

Throughout this paper, we shall classify the $132$-avoiding permutations 
$\sg = \sg_1 \ldots \sg_n$ by the position of $n$ 
in $\sg$. That is, let 
$S^{(i)}_n(132)$ denote the set of $\sg \in S_n(132)$ such 
that $\sg_i =n$. Clearly each $\sg \in  S_n^{(i)}(132)$ has the structure 
pictured in Figure \ref{fig:basic2}. That is, in the graph of 
$\sg$, the elements to the left of $n$, $A_i(\sg)$, have 
the structure of a $132$-avoiding permutation, the elements 
to the right of $n$, $B_i(\sg)$, have the structure of a 
$132$-avoiding permutation, and all the elements in 
$A_i(\sg)$ lie above all the elements in 
$B_i(\sg)$.  It is well-known that the number of $132$-avoiding 
permutations in $S_n$ is the {\em Catalan number} 
$C_n = \frac{1}{n+1} \binom{2n}{n}$ and the generating 
function for the $C_n$'s is given by 
\begin{equation*}\label{Catalan}
C(t) = \sum_{n \geq 0} C_n t^n = \frac{1-\sqrt{1-4t}}{2t}=
\frac{2}{1+\sqrt{1-4t}}.
\end{equation*}

\fig{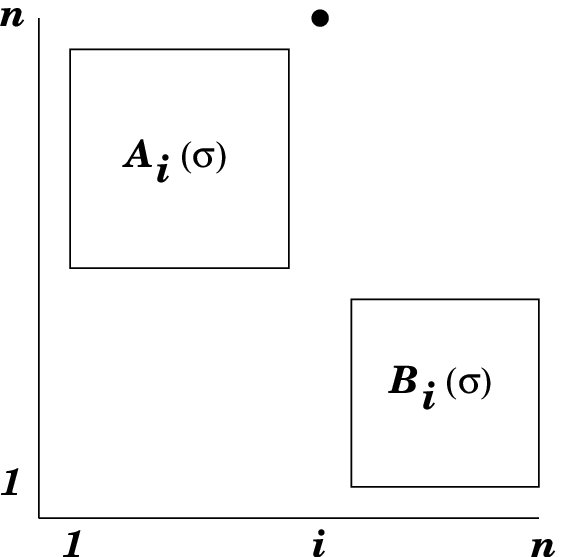}{The structure of $132$-avoiding permutations.}

Suppose that $n \geq \ell$.  
It is clear that $n$ can never match 
the pattern $\MMP(k,\ell,m,0)$ for $k,m \geq 1$ in any 
$\sg \in S_n(132)$.    
For $1 \leq i \leq n$,  it is easy to see that as we sum 
over all the permutations $\sg$ in $S_n^{(i)}(132)$, our choices 
for the structure for $A_i(\sg)$ will contribute a factor 
of $Q_{i-1,132}^{(k-1,\ell,m,0)}(x)$ to $Q_{n,132}^{(k,\ell,m,0)}(x)$. 
Similarly, our choices 
for the structure for $B_i(\sg)$ will contribute a factor 
of $Q_{n-i,132}^{(k,\ell-i,m,0)}(x)$ to $Q_{n,132}^{(k,\ell,m,0)}(x)$ if 
$i < \ell$ since $\sg_1 \ldots \sg_i$ will automatically be 
in the second quadrant relative to the coordinate system 
with the origin at $(s,\sg_s)$ for any $s > i$.  However if $i \geq \ell$, 
then our choices 
for the structure for $B_i(\sg)$ will contribute a factor 
of $Q_{n-i,132}^{(k,0,m,0)}(x)$ to $Q_{n,132}^{(k,\ell,m,0)}(x)$. It follows that for $n \geq \ell$, 
\begin{equation*}
Q_{n,132}^{(k,\ell,m,0)}(x) = \sum_{i=1}^{\ell -1} 
Q_{i-1,132}^{(k-1,\ell,m,0)}(x)Q_{n-i,132}^{(k,\ell-i,m,0)}(x) + 
\sum_{i=\ell}^{n} Q_{i-1,132}^{(k-1,\ell,m,0)}(x)
Q_{n-i,132}^{(k,0,m,0)}(x).
\end{equation*}
Note that for $i < \ell$, $Q_{i-1,132}^{(k-1,\ell,m,0)}(x) =C_{i-1}$. 
Thus, for $n \geq \ell$,
\begin{equation}\label{Q-klm02}
Q_{n,132}^{(k,\ell,m,0)}(x) =\\ \sum_{i=1}^{\ell -1} 
C_{i-1}Q_{n-i,132}^{(k,\ell-i,m,0)}(x) + 
\sum_{i=\ell}^{n} Q_{i-1,132}^{(k-1,\ell,m,0)}(x)
Q_{n-i,132}^{(k,0,m,0)}(x).
\end{equation}

Multiplying both sides of (\ref{Q-klm02}) by $t^n$ and 
summing for $n \geq \ell$, we 
see that for $k, \ell \geq 1$, 
\begin{eqnarray*} 
Q_{132}^{(k,\ell,m,0)}(t,x) &=& \sum_{j=0}^{\ell -1} C_jt^j+
\sum_{i=1}^{\ell -1} C_{i-1}t^i \sum_{u \geq \ell -i}
Q_{u,132}^{(k,\ell-i,m,0)}(x)t^u +  \\
&& t \sum_{n \geq \ell} 
\sum_{i=1}^n  Q_{i-1,132}^{(k-1,\ell,m,0)}(x)t^{i-1} 
Q_{n-i,132}^{(k,0,m,0)}(x)t^{n-i} \nonumber \\
&=& \sum_{j=0}^{\ell -1} C_jt^j+ \sum_{i=1}^{\ell -1} C_{i-1}t^i 
\left( Q_{132}^{(k,\ell-i,m,0)}(t,x) - 
\sum_{j=0}^{\ell -i -1}C_jt^j\right) +  \nonumber \\
&& t Q_{132}^{(k,0,m,0)}(t,x)\left(  
Q_{132}^{(k-1,\ell,m,0)}(t,x)- \sum_{s=0}^{\ell -2} C_s t^s\right)  
\nonumber \\
&=& C_{\ell -1} t^{\ell -1} + t Q_{132}^{(k,0,m,0)}(t,x)Q_{132}^{(k-1,\ell,m,0)}(t,x) + \nonumber \\
&&\sum_{s =0}^{\ell -2}C_s t^s \left(1+  t Q_{132}^{(k,\ell-1-s,m,0)}(t,x) -
tQ_{132}^{(k,0,m,0)}(t,x) - t \sum_{j=0}^{\ell -2 -s} C_jt^j\right). \nonumber
\end{eqnarray*}

Thus, we have the following theorem. 
\begin{theorem}\label{thm:Qklm0}
\begin{multline}\label{eq:Qklm0}
Q_{132}^{(k,\ell,m,0)}(t,x) = C_{\ell -1}t^{\ell -1}+  
t Q_{132}^{(k,0,m,0)}(t,x)Q_{132}^{(k-1,\ell,m,0)}(t,x) + \\
\sum_{s =0}^{\ell -2}C_s t^s \left(1+ t Q_{132}^{(k,\ell-1-s,m,0)}(t,x) -
tQ_{132}^{(k,0,m,0)}(t,x) - t \sum_{j=0}^{\ell -2 -s} C_jt^j\right).
\end{multline}
\end{theorem}

Note that since we can compute 
$Q_{132}^{(k,0,m,0)}(t,x)$ by Theorem \ref{thm:Qk0l0} and 
$Q_{132}^{(0,\ell,m,0)}(t,x)$ by Theorem \ref{thm:Q0kl0}, 
we can use (\ref{eq:Qklm0}) to compute $Q_{132}^{(k,\ell,m,0)}(t,x)$ 
for any $k,\ell,m \geq 1$. 

\subsection{Explicit formulas for  $Q^{(k,\ell,m,0)}_{n,132}(x)|_{x^r}$}

It follows from  Theorem \ref{thm:Qklm0} that  
\begin{equation}\label{Qklm0gf0}
Q_{132}^{(k,1,m,0)}(t,x) = 1+ 
t Q_{132}^{(k,0,m,0)}(t,x)Q_{132}^{(k-1,1,m,0)}(t,x)
\end{equation}  
and 
\begin{equation*}\label{Qk2m0gf00}
Q_{132}^{(k,2,m,0)}(t,x) = 1+ 
t Q_{132}^{(k,0,m,0)}(t,x)(Q_{132}^{(k-1,2,m,0)}(t,x)-1) + 
tQ_{132}^{(k,1,m,0)}(t,x).
\end{equation*}

Note that it follows from Theorems \ref{thm:Qk0l0} and \ref{thm:Q0kl0} that 
\begin{eqnarray*}
Q_{132}^{(1,1,1,0)}(t,0) &=&  1 + t Q_{132}^{(1,0,1,0)}(t,0)Q_{132}^{(0,1,1,0)}(t,0)\\
&=& 1+ t\frac{1-t}{1-2t}\frac{1-t}{1-2t}= 
\frac{1-3t+2t^2+t^3}{(1-2t)^2}.  
\end{eqnarray*}
Thus, the generating function of the sequence 
$(Q_{n,132}^{(1,1,1,0)}(0))_{n \geq 1}$ is 
$\left(\frac{1-t}{1-2t} \right)^2$ which is the generating 
function of the sequence A045623 in the OEIS.  The $n$-th term $a_n$ of this sequence 
has many combinatorial interpretations including the number of 
$1$s in all partitions of $n+1$ and the number of 132-avoiding permutations 
of $S_{n+2}$ which contain exactly one occurrence of the pattern 
213.  We note that for a permutation $\sg$ to avoid the pattern 
$\MMP(1,1,1,0)$, it must simultaneously avoid the patterns 
3124, 4123, 1324, and 1423. Thus, the number of permutations 
$\sg \in S_n(132)$ which avoid $\MMP(1,1,1,0)$ is the number 
of permutations in $S_n$ that simultaneously avoid the patterns 
132, 3124, and 4123.

\begin{problem} Find simple bijections between the set of 
permutations $\sg \in S_n(132)$ which avoid  
$\MMP(1,1,1,0)$ and the other combinatorial interpretations of 
the sequence A045623 in the OEIS.
\end{problem}

Note that it follows from Theorem \ref{thm:Qk0l0} and our previous results that  
\begin{eqnarray*}
Q_{132}^{(2,1,1,0)}(t,0) &=&  1 + t Q_{132}^{(2,0,1,0)}(t,0)Q_{132}^{(1,1,1,0)}(t,0) \\
&=& 1+t\frac{1-2t}{1-3t+t^2}\left(\frac{1-3t+2t^2+t^3}{(1-2t)^2}\right)\\
&=& \frac{1-4t+4t^2+t^4}{1-5t+7t^2-2t^3}.
\end{eqnarray*}
The sequence  $(Q_{n,132}^{(2,1,1,0)}(0))_{n \geq 1}$ is the sequence  
A142586 in the OIES which has the generating function 
$\frac{1-3t+2t^2+t^3}{(1-3t+t^2)(1-2t)}$. That is, 
$\frac{1-4t+4t^2+t^4}{1-5t+7t^2-2t^3}-1 = \frac{t(1-3t+2t^2+t^3)}{(1-3t+t^2)(1-2t)}$. This sequence has no listed combinatorial interpretation 
so that we have found a combinatorial interpretation of this sequence. 

Similarly, 
\begin{eqnarray*}
Q_{132}^{(3,1,1,0)}(t,0) &=&  1 + t Q_{132}^{(3,0,1,0)}(t,0)Q_{132}^{(2,1,1,0)}(t,0)\\
&=& 1+t\frac{1-3t+t^2}{1-4t+3t^2}\ \frac{1-4t+4t^2+t^4}{1-5t+7t^2-2t^3}\\
&=& \frac{1-5t+7t^2-2t^3+t^5}{1-6t+11t^2-6t^3}.
\end{eqnarray*}

\vspace{-0.5cm}

\begin{eqnarray*}
Q_{132}^{(1,1,2,0)}(t,0) &=&  1 + t Q_{132}^{(1,0,2,0)}(t,0)Q_{132}^{(0,1,2,0)}(t,0)\\
&=& 1+t\frac{1-t-t^2}{1-2t-t^2}\ \frac{1-t-t^2}{1-2t-t^2}\\ 
&=& \frac{1-3t+3t^3+3t^4+t^5}{(1-2t-t^2)^2}.
\end{eqnarray*}

\vspace{-0.5cm}

\begin{eqnarray*}
Q_{132}^{(2,1,2,0)}(t,0) &=&  1 + t Q_{132}^{(2,0,2,0)}(t,0)Q_{132}^{(1,1,2,0)}(t,0)\\
&=& 1+t\frac{1-2t-t^2}{1-3t+t^3}\ \frac{1-3t+3t^3+3t^4+t^5}{(1-2t-t^2)^2}\\ 
&=& \frac{1-4t+2t^2+4t^3+t^4+2t^5+t^6}{(1-2t-t^2)(1-3t+t^3)}.
\end{eqnarray*}

Using (\ref{Qklm0gf0}) and Theorem \ref{thm:Qk0l0}, 
we have computed the following.

\begin{align*}
&Q_{132}^{(1,1,1,0)}(t,x) = 1+t+2 t^2+5 t^3+(12+2 x) t^4+\left(28+12 x+2 x^2\right) t^5+\ \ \ \ \ \ \ \ \ \ \ \\
&\left(64+48 x+18 x^2+2 x^3\right) t^6+\left(144+160 x+97 x^2+26 x^3+2 x^4\right)
t^7+\\
&\left(320+480 x+408 x^2+184 x^3+36 x^4+2 x^5\right) t^8+\\
&\left(704+1344 x+1479 x^2+958 x^3+327 x^4+48 x^5+2 x^6\right) t^9+\cdots .
\end{align*}

\vspace{-0.5cm}

\begin{align*}
&Q_{132}^{(1,1,2,0)}(t,x) = 1+t+2 t^2+5 t^3+14 t^4+(38+4 x) t^5+\left(102+26 x+4 x^2\right) t^6+\\
&\left(271+120 x+34 x^2+4 x^3\right) t^7+\left(714+470 x+200 x^2+42
x^3+4 x^4\right) t^8+\\
&\left(1868+1672 x+964 x^2+304 x^3+50 x^4+4 x^5\right) t^9+\cdots .
\end{align*}

\vspace{-0.5cm}

\begin{align*}
&Q_{132}^{(1,1,3,0)}(t,x) = 1+t+2 t^2+5 t^3+14 t^4+42 t^5+(122+10 x) t^6+
\ \ \ \ \ \ \ \ \ \ \ \ \ \ \ \ \  \\
&\left(351+68 x+10 x^2\right) t^7+\left(1006+326 x+88 x^2+10 x^3\right) t^8+\\
&\left(2868+1364
x+512 x^2+108 x^3+10 x^4\right) t^9+ \cdots .
\end{align*}

We can explain the highest and second highest coefficients of 
$x$ in these series.  That is, we have the following theorem.

\begin{theorem} \ 
\begin{itemize}
\item[(i)] For all $m \geq 1$ and $n \geq 3+m$, the highest power 
of $x$ that occurs in $Q_{n,132}^{(1,1,m,0)}(x)$ is $x^{n-2-m}$ which 
appears with a coefficient of $2C_m$. 

\item[(ii)] For $n \geq 5$, 
$Q_{n,132}^{(1,1,1,0)}(x)|_{x^{n-4}} = 6 + 2\binom{n-2}{2}$.

\item[(iii)] For $m \geq 2$ and $n \geq 4+ m$
$Q_{n,132}^{(1,1,m,0)}(x)|_{x^{n-3-m}} = 2C_{m+1}+ 8C_m  + 4C_m(n-4)$.  
\end{itemize}
\end{theorem}
\begin{proof} 

It is easy to see that for the maximum number 
of $\MMP(1,1,m,0)$-matches in a $\sg \in S_n(132)$, the permutation 
must be of the form 
$(n-1)\tau  (m+1) \ldots (n-2) n$ or  $n\tau  (m+1) \ldots (n-2) (n-1)$
where $\tau \in S_m(132)$. Thus, the highest power of 
$x$ occurring in  $Q_{n,132}^{(1,1,m,0)}(x)$ is $x^{n-2-m}$ which 
occurs with a coefficient of $2C_m$.

For parts (ii) and (iii), we have the recursion that 

\begin{equation}\label{eq:2high11m0a}
Q_{n,132}^{(1,1,m,0)}(x) = \sum_{i=1}^n Q_{i-1,132}^{(0,1,m,0)}(x)
Q_{n-i,132}^{(1,0,m,0)}(x).
\end{equation}

We proved in \cite{kitremtieII} that  the highest power of 
$x$ which occurs in either  $Q_{n,132}^{(0,1,m,0)}(x)$ or 
$Q_{n,132}^{(1,0,m,0)}(x)$ is $x^{n-1-m}$ and 
\begin{equation*}\label{eq:2high11m0b}
Q_{n,132}^{(0,1,m,0)}(x)|_{x^{n-1-m}} =  
Q_{n,132}^{(1,0,m,0)}(x)|_{x^{n-1-m}} = C_m.
\end{equation*}

It is then easy to check that the highest coefficient of $x$ 
in $Q_{i-1,132}^{(0,1,m,0)}(x)
Q_{n-i,132}^{(1,0,m,0)}(x)$ is less than $x^{n-3-m}$ for 
$i=3, \ldots, n-3$. 

We also proved in \cite{kitremtieII} that  
\begin{eqnarray*}
Q_{n,132}^{(1,0,1,0)}(x)|_{x^{n-3}} &=& Q_{n,132}^{(0,1,1,0)}(x)|_{x^{n-3}}
= 2 + \binom{n-1}{2} \ \mbox{for}\  n \geq 4 \ \mbox{and} \\
Q_{n,132}^{(1,0,m,0)}(x)|_{x^{n-m-2}} &=& 
Q_{n,132}^{(0,1,m,0)}(x)|_{x^{n-m-2}} \\
&=& C_{m+1}+C_m+2C_m(n-2-m) \ \mbox{for}\  n \geq 3+m \ \mbox{and} \ m \geq 2.
\end{eqnarray*}

For $m =1$, we are left with 4 cases to consider in 
the recursion (\ref{eq:2high11m0a}). We start with the $m=1$ case. \\
\ \\
{\bf Case 1.} $i=1$. In this case, $Q_{i-1,132}^{(0,1,1,0)}(x)
Q_{n-i,132}^{(1,0,1,0)}(x)|_{x^{n-4}} = Q_{n-1,132}^{(1,0,1,0)}(x)|_{x^{n-4}}$ 
and 
$$Q_{n-1,132}^{(1,0,1,0)}(x)|_{x^{n-4}} = 
2+\binom{n-2}{2} \ \mbox{for } n \geq 5.$$
{\bf Case 2.} $i=2$. In this case, $Q_{i-1,132}^{(0,1,1,0)}(x)
Q_{n-i,132}^{(1,0,1,0)}(x)|_{x^{n-4}} = Q_{n-2,132}^{(1,0,1,0)}(x)|_{x^{n-4}}$ 
and 
$$Q_{n-2,132}^{(1,0,1,0)}(x)|_{x^{n-4}} = 1 \ \mbox{for } n \geq 5.$$
{\bf Case 3.} $i=n-1$. In this case, $Q_{i-1,132}^{(0,1,1,0)}(x)
Q_{n-i,132}^{(1,0,1,0)}(x)|_{x^{n-4}} = Q_{n-2,132}^{(0,1,1,0)}(x)|_{x^{n-4}}$ 
and 
$$Q_{n-2,132}^{(0,1,1,0)}(x)|_{x^{n-4}} = 1 \ \mbox{for } n \geq 5.$$
{\bf Case 4.} $i=n$. In this case, $Q_{i-1,132}^{(0,1,1,0)}(x)
Q_{n-i,132}^{(1,0,1,0)}(x)|_{x^{n-4}} = Q_{n-1,132}^{(0,1,1,0)}(x)|_{x^{n-4}}$ 
and 
$$Q_{n-1,132}^{(0,1,1,0)}(x)|_{x^{n-4}} = 
2+\binom{n-2}{2} \ \mbox{for } n \geq 5.$$

Thus, $Q_{n,132}^{(0,1,1,0)}(x)|_{x^{n-4}} =6+2\binom{n-2}{2}$ for $n \geq 5$.

Next we consider the case when $m\geq 2$. Again we have 4 cases.\\ 
\ \\
{\bf Case 1.} $i=1$. In this case, $Q_{i-1,132}^{(0,1,m,0)}(x)
Q_{n-i,132}^{(1,0,m,0)}(x)|_{x^{n-3-m}} = 
Q_{n-1,132}^{(1,0,m,0)}(x)|_{x^{n-3-m}}$ 
and 
$$Q_{n-1,132}^{(1,0,m,0)}(x)|_{x^{n-3-\ell}} = 
C_{m+1}+C_m +2C_m(n -3 -m) \ \mbox{for } n \geq 4+m.$$
{\bf Case 2.} $i=2$. In this case, $Q_{i-1,132}^{(0,1,m,0)}(x)
Q_{n-i,132}^{(1,0,m,0)}(x)|_{x^{n-3-m}} = 
Q_{n-2,132}^{(1,0,m,0)}(x)|_{x^{n- 3-m}}$ 
and 
$$Q_{n-2,132}^{(1,0,m,0)}(x)|_{x^{n-3-m}} =C_m  \ \mbox{for } n \geq 4+m.$$
{\bf Case 3.} $i=n-1$. In this case, $Q_{i-1,132}^{(0,1,m,0)}(x)
Q_{n-i,132}^{(1,0,m,0)}(x)|_{x^{n-3-m}} = 
Q_{n-2,132}^{(0,1,m,0)}(x)|_{x^{n-3-m}}$ 
and 
$$Q_{n-2,132}^{(0,1,m,0)}(x)|_{x^{n-3-m}} = C_m \ \mbox{for } n \geq 4+m.$$
{\bf Case 4.} $i=n$. In this case, $Q_{i-1,132}^{(0,1,2,0)}(x)
Q_{n-i,132}^{(1,0,m,0)}(x)|_{x^{n-3-m}} = 
Q_{n-1,132}^{(0,1,m,0)}(x)|_{x^{n-3-m}}$ 
and 
$$Q_{n-1,132}^{(0,1,m,0)}(x)|_{x^{n-3-m}} = C_{m+1}+C_m +2C_m(n -3 -m) \ \mbox{for } n \geq 4+m.$$

Thus, for $n \geq 4+m$, 
\begin{eqnarray*} 
Q_{n,132}^{(1,1,m,0)}(x)|_{x^{n-3-m}} &=&  2C_{m+1}+4C_m +4C_m(n -3 -m) \\
&=&  2C_{m+1} +8C_m + 4C_m(n -4 -m). 
\end{eqnarray*}

Thus, when $m=2$, we obtain that 
$$Q_{n,132}^{(0,1,2,0)}(x)|_{x^{n-5}} =26+8(n-6) \ \mbox{for } n \geq 6$$ 
and, for $m=3$, we obtain that  
$$Q_{n,132}^{(0,1,3,0)}(x)|_{x^{n-6}} =68 + 20(n-7) \mbox{for } n \geq 7$$ 
which agrees with our computed series. 
\end{proof}

We also have computed that 

\begin{align*}
&Q_{132}^{(2,1,1,0)}(t,x) =1+t+2 t^2+5 t^3+14 t^4+(39+3 x) t^5+\left(107+22 x+3 x^2\right) t^6+\ \ \ \ \ \ \ \ \ \ \ \\
&\left(290+105 x+31 x^2+3 x^3\right) t^7+\left(779+415 x+190 x^2+43
x^3+3 x^4\right) t^8+\\
&\left(2079+1477 x+909 x^2+336 x^3+58 x^4+3 x^5\right) t^9+\\
&\left(5522+4922 x+3765 x^2+1938 x^3+570 x^4+76 x^5+3 x^6\right) t^{10}+
\cdots ,
\end{align*}

\vspace{-0.5cm}

\begin{align*}
&Q_{132}^{(2,1,2,0)}(t,x) =1+t+2 t^2+5 t^3+14 t^4+42 t^5+(126+6 x) t^6+\left(376+47 x+6 x^2\right) t^7+\\
&\left(1115+250 x+59 x^2+6 x^3\right) t^8+ \left(3289+1110 x+386
x^2+71 x^3+6 x^4\right) t^9+\\
&\left(9660+4444 x+2045 x^2+558 x^3+83 x^4+6 x^5\right) t^{10}+
\cdots , \ \mbox{and} 
\end{align*}

\vspace{-0.5cm}

\begin{align*}
&Q_{132}^{(2,1,3,0)}(t,x) =1+t+2 t^2+5 t^3+14 t^4+42 t^5+132 t^6+(414+15 x) t^7+\ \ \ \ \ \ \ \ \ \  \ \ \ \ \ \ \ \ \\
&\left(1293+122 x+15 x^2\right) t^8+
\left(4025+670 x+152 x^2+15 x^3\right) t^9+\\
&\left(12486+3124
x+989 x^2+182 x^3+15 x^4\right) t^{10}+ \cdots .
\end{align*}

Again one can easily explain the highest coefficient in 
$Q_{n,132}^{(2,1,m,0)}(x)$. That is, to have the maximum number 
of $\MMP(2,1,m,0)$-matches in a $\sg \in S_n(132)$, the permutation 
must be of the form 
\begin{eqnarray*}
&&(n-2)\tau  (m+1) \ldots (n-3)(n-1) n, \\
&&(n-1)\tau  (m+1) \ldots (n-3)(n-2) n, \mbox{or} \\ 
&&n\tau  (m+1) \ldots (n-3) (n-2) (n-1)
\end{eqnarray*}
where $\tau \in S_m(132)$. Thus, the highest power of 
$x$ occurring in  $Q_{n,132}^{(2,1,m,0)}(x)$ is $x^{n-3-m}$ which 
occurs with a coefficient of $3C_m$.

We have computed that 
\begin{align*}
&Q_{132}^{(1,2,1,0)}(t,x) = 1+t+2 t^2+5 t^3+14 t^4+(37+5 x) t^5+\left(94+33 x+5 x^2\right) t^6+\\
&\left(232+144 x+48 x^2+5 x^3\right) t^7+\left(560+520 x+277 x^2+68
x^3+5 x^4\right) t^8+\\
&\left(1328+1680 x+1248 x^2+508 x^3+93 x^4+5 x^5\right) t^9+
\cdots .
\end{align*}

\vspace{-0.5cm}

\begin{align*}
&Q_{132}^{(1,2,2,0)}(t,x) = 1+t+2 t^2+5 t^3+14 t^4+42 t^5+(122+10 x) t^6+
\ \ \ \ \ \ \ \ \ \ \ \ \ \ \ \\
&\left(348+71 x+10 x^2\right) t^7+\left(978+351 x+91 x^2+10 x^3\right) t^8+\\
&\left(2715+1463x+563 x^2+111 x^3+10 x^4\right) t^9+
\cdots ,  \ \mbox{and}
\end{align*}

\vspace{-0.5cm}

\begin{align*}
&Q_{132}^{(1,2,3,0)}(t,x) = 1+t+2 t^2+5 t^3+14 t^4+42 t^5+132 t^6+(404+25 x) t^7+ \ \ \ \ \\
&\left(1220+185 x+25 x^2\right) t^8+
\left(3655+947 x+235 x^2+25 x^3\right) t^9 + 
\cdots . 
\end{align*}

Again, one can easily explain the highest coefficient in 
$Q_{n,132}^{(1,2,m,0)}(x)$. That is, to have the maximum number 
of $\MMP(1,2,m,0)$-matches in a $\sg \in S_n(132)$, one 
must be of the form 
\begin{eqnarray*}
&&(n-2)(n-1)\tau  (m+1) \ldots (n-3)n,\\
&&(n-1)(n-2)\tau  (m+1) \ldots (n-3)n,\\
&&n(n-2)\tau  (m+1) \ldots (n-3)(n-1),\\
&&n(n-1)\tau  (m+1) \ldots (n-3)(n-2), \ \mbox{or} \\
&&(n-1)n\tau  (m+1) \ldots (n-3)(n-2)
\end{eqnarray*}
where $\tau \in S_m(132)$.
Thus, the highest power of 
$x$ occurring in  $Q_{n,132}^{(1,2,m,0)}(x)$ is $x^{n-3-m}$ which 
occurs with a coefficient of $5C_m$. 

Finally, we have computed that 
\begin{align*}
&Q_{132}^{(2,2,1,0)}(t,x) = 1+t+2 t^2+5 t^3+14 t^4+42 t^5+(123+9 x) t^6+
\left(351+69 x+9 x^2\right) t^7+\\
&\left(982+343 x+96 x^2+9 x^3\right) t^8+\left(2707+1405 x+609
x^2+132 x^3+9 x^4\right) t^9+
\cdots , 
\end{align*}

\vspace{-0.5cm}

\begin{align*}
&Q_{132}^{(2,2,2,0)}(t,x) = 1+t+2 t^2+5 t^3+14 t^4+42 t^5+132 t^6+(411+18 x)t^7 +\ \ \ \ \ \ \ \ \ \ \ \ \ \ \ \ \ \ \\
& \left(1265+147 x+18 x^2\right)t^8+ \left(3852+809 x+183 x^2+18 x^3\right)t^9+\\
&\left(11626+3704 x+1229 x^2+219 x^3+18 x^4\right)t^{10}+ \cdots , \ \mbox{and}  
\end{align*}

\vspace{-0.5cm}

\begin{align*}
&Q_{132}^{(2,2,3,0)}(t,x) = 1+t+2 t^2+5 t^3+14 t^4+42 t^5+132 t^6+429 t^7+
(1385+45 x) t^8 +\ \ \ \ \ \ \\
& \left(4436+381 x+45 x^2\right)t^9+ \left(14118+2162 x+471 x^2+45
x^3\right)t^{10} +\\
& \left(44670+10361 x+3149 x^2+561 x^3+45 x^4\right)t^{11}+ \cdots .
\end{align*}

Again, one can easily explain the highest coefficient in 
$Q_{n,132}^{(2,2,m,0)}(x)$. That is, to have the maximum number 
of $\MMP(2,2,m,0)$-matches in a $\sg \in S_n(132)$, one 
must be of the form 
\begin{eqnarray*}
&&n(n-1)\tau  (m+1) \ldots (n-4)(n-3)(n-2),\\
&&(n-1)n\tau  (m+1) \ldots (n-4)(n-3)(n-2),\\
&&n(n-2)\tau  (m+1) \ldots (n-4)(n-3)(n-1),\\
&&n(n-3)\tau  (m+1) \ldots (n-4)(n-2)(n-1),\\
&&(n-1)(n-2)\tau  (m+1) \ldots (n-4)(n-3)n,\\
&&(n-2)(n-1)\tau  (m+1) \ldots (n-4)(n-3)n,\\
&&(n-1)(n-3)\tau  (m+1) \ldots (n-4)(n-2)n,\\
&&(n-2)(n-3)\tau  (m+1) \ldots (n-4)(n-1)n, \ \mbox{or} \\
&&(n-3)(n-2)\tau  (m+1) \ldots (n-4)(n-1)n
\end{eqnarray*}
where $\tau \in S_m(132)$.
Thus, the highest power of 
$x$ occurring in  $Q_{n,132}^{(2,2,m,0)}(x)$ is $x^{n-4-m}$ which 
occurs with a coefficient of $9C_m$.

\section{$Q_{n,132}^{(0,k,\ell,m)}(x) = Q_{n,132}^{(0,m,\ell,k)}(x)$ 
where $k,\ell,m \geq 1$}

By Lemma \ref{sym}, we only need to consider $Q_{n,132}^{(0,k,\ell,m)}(x)$.  Suppose that $k, \ell, m \geq 1$ and $n \geq k+m$.  
It is clear that $n$ can never match 
the pattern $\MMP(0,k,\ell,m)$ for $k,\ell, m \geq 1$ in any 
$\sg \in S_n(132)$. If $\sg = \sg_1 \ldots \sg_n \in S_n(132)$  
and $\sg_i =n$, then we have three cases, depending 
on the value of $i$. \\
\ \\
{\bf Case 1.} $i < k$. It is easy to see that as we sum 
over all the permutations $\sg$ in $S_n^{(i)}(132)$, our choices 
for the structure for $A_i(\sg)$ will contribute a factor 
of $C_{i-1}$ to $Q_{n,132}^{(0,k,\ell,m)}(x)$ since 
none of the elements $\sg_j$ for $j \leq k$ can match 
$\MMP(0,k,\ell,m)$ in $\sg$. 
Similarly, our choices 
for the structure for $B_i(\sg)$ will contribute a factor 
of $Q_{n-i,132}^{(0,k-i,\ell,m)}(x)$ to $Q_{n,132}^{(0,k,\ell,m)}(x)$ 
since $\sg_1 \ldots \sg_i$ will automatically be 
in the second quadrant relative to the coordinate system 
with the origin at $(s,\sg_s)$ for any $s > i$.  Thus, the 
permutations in Case 1 will contribute 
$$\sum_{i=1}^{k-1} C_{i-1} Q_{n-i,132}^{(0,k-i,\ell,m)}(x)$$ 
to $Q_{n,132}^{(0,k,\ell,m)}(x)$.\\
\ \\
{\bf Case 2.} $k \leq i \leq n-m$. It is easy to see that as we sum 
over all the permutations $\sg$ in $S_n^{(i)}(132)$, our choices 
for the structure for $A_i(\sg)$ will contribute a factor 
of $Q_{i-1,132}^{(0,k,\ell,0)}(x)$ to $Q_{n,132}^{(0,k,\ell,m)}(x)$ since 
the elements in $B_i(\sg)$ will all be in the fourth quadrant 
relative to a coordinate system centered at $(r,\sg_r)$ for 
$r \leq i$ in this case. 
Similarly, our choices 
for the structure for $B_i(\sg)$ will contribute a factor 
of $Q_{n-i,132}^{(0,0,\ell,m)}(x)$ to $Q_{n,132}^{(0,k,\ell,m)}(x)$ 
since $\sg_1 \ldots \sg_i$ will automatically be 
in the second quadrant relative to the coordinate system 
with the origin at $(s,\sg_s)$ for any $s > i$.  Thus, the 
permutations in Case 2 will contribute 
$$\sum_{i=k}^{n-m} Q_{i-1,132}^{(0,k,\ell,0)}(x) 
Q_{n-i,132}^{(0,0,\ell,m)}(x)$$ 
to $Q_{n,132}^{(0,k,\ell,m)}(x)$.\\
\ \\
{\bf Case 3.} $i \geq n-m +1$. It is easy to see that as we sum 
over all the permutations $\sg$ in $S_n^{(i)}(132)$, our choices 
for the structure for $A_i(\sg)$ will contribute a factor 
of $Q_{i-1,132}^{(0,k,\ell,m-(n-i))}(x)$ 
to $Q_{n,132}^{(0,k,\ell,m)}(x)$ since 
the elements in $B_i(\sg)$ will all be in the fourth quadrant 
relative to a coordinate system centered at $(r,\sg_r)$ for 
$r \leq i$ in this case. 
Similarly, our choices 
for the structure for $B_i(\sg)$ will contribute a factor 
of $C_{n-i}$ to $Q_{n,132}^{(0,k,\ell,m)}(x)$ 
since the elements in $B_i(\sg)$ do not have 
enough elements to the right to match $\MMP(0,k,\ell,m)$ in 
$\sg$. Thus, the 
permutations in Case 3 will contribute 
$$\sum_{i=n-m+1}^{n} Q_{i-1,132}^{(0,k,\ell,m-(n-i))}(x) 
C_{n-i}$$ 
to $Q_{n,132}^{(0,k,\ell,m)}(x)$. Hence, for $ n \geq k + m$,  
\begin{eqnarray}\label{0klm-rec} 
Q_{n,132}^{(0,k,\ell,m)}(x) &=& \sum_{i=1}^{k-1} C_{i-1} 
Q_{n-i,132}^{(0,k-i,\ell,m)}(x) + 
\sum_{i=k}^{n-m} Q_{i-1,132}^{(0,k,\ell,0)}(x) 
Q_{n-i,132}^{(0,0,\ell,m)}(x) +\nonumber \\
&& \sum_{i=n-m+1}^{n} Q_{i-1,132}^{(0,k,\ell,m-(n-i))}(x) C_{n-i}.
\end{eqnarray}
Multiplying (\ref{0klm-rec}) by $t^n$ and summing, it is easy 
to compute that 
\begin{eqnarray*}\label{eq:Q0klm0}
Q_{132}^{(0,k,\ell,m)}(t,x) &=& \sum_{p=0}^{k+m-1} C_p t^p + \nonumber \\
&&\sum_{i=0}^{k-2} C_i t^i \left( t Q_{132}^{(0,k-1-i,\ell,m)}(t,x) -t 
\sum_{r=0}^{k-i+m-2} C_rt^r \right) + \nonumber \\
&&t \left( Q_{132}^{(0,k,\ell,0)}(t,x) - \sum_{a=0}^{k-2} C_a t^a \right) 
\left( Q_{132}^{(0,0,\ell,m)}(t,x) - \sum_{b=0}^{m-1} C_b t^b \right) + 
\nonumber \\
&& \sum_{j=0}^{m-1} C_j t^j \left( t Q_{132}^{(0,k,\ell,m-j)}(t,x) -t 
\sum_{s=0}^{k+m-j-2} C_st^s\right).
\end{eqnarray*}
Note that the $j=0$ term in the last sum 
is $tQ_{132}^{(0,k,\ell,m)}(t,x) - t\sum_{s=0}^{k+m-2}C_st^s$. 
Thus, taking the term $tQ_{132}^{(0,k,\ell,m)}(t,x)$ over to the other 
side and combining the sum $t\sum_{s=0}^{k+m-2}C_st^s$ with 
the sum $\sum_{p=0}^{k+m-1}C_pt^p$ to obtain 
$C_{k+m-1}t^{k+m-1}+(1-t)\sum_{p=0}^{k+m-2}C_pt^p$ and then 
dividing both sides 
by $1-t$ will yield the following theorem. 
\begin{theorem}\label{thm:Q0klm}
\begin{eqnarray*}\label{eq:Q0klm}
Q_{132}^{(0,k,\ell,m)}(t,x) &=& \sum_{p=0}^{k+m-2} C_p t^p + 
\frac{C_{k+m-1}t^{k+m-1}}{1-t} + \nonumber \\
&&\frac{t}{1-t}\sum_{i=0}^{k-2} C_i t^i \left( Q_{132}^{(0,k-1-i,\ell,m)}(t,x) 
-\sum_{r=0}^{k-i+m-2} C_rt^r\right) + \nonumber \\
&&\frac{t}{1-t} \left( Q_{132}^{(0,k,\ell,0)}(t,x) - \sum_{a=0}^{k-2} C_a t^a \right) 
\left( Q_{132}^{(0,0,\ell,m)}(t,x) - \sum_{b=0}^{m-1} C_b t^b \right) + 
\nonumber \\
&& \frac{t}{1-t}\sum_{j=1}^{m-1} C_j t^j \left( Q_{132}^{(0,k,\ell,m-j)}(t,x) 
- \sum_{s=0}^{k+m-j-2} C_st^s\right).
\end{eqnarray*}
\end{theorem}

Note that since we can compute $Q_{132}^{(0,k,\ell,0)}(t,x) = 
Q_{132}^{(0,0,\ell,k)}(t,x)$ by Theorem \ref{thm:Q0kl0}, we 
can compute $Q_{132}^{(0,k,\ell,m)}(t,x)$ for 
all $k,\ell,m \geq 1$.

\subsection{Explicit formulas for  $Q^{(0,k,\ell,m)}_{n,132}(x)|_{x^r}$}

It follows from Theorem \ref{thm:Q0klm} that 
 
\begin{eqnarray*}
Q_{132}^{(0,1,\ell,1)}(t,x) &=& 1 +\frac{t}{1-t} + 
\frac{t}{1-t}Q_{132}^{(0,1,\ell,0)}(t,x)
(Q_{132}^{(0,0,\ell,1)}(t,x)-1)  \nonumber \\
&=& \frac{1}{1-t} + \frac{t}{1-t}Q_{132}^{(0,1,\ell,0)}(t,x)
(Q_{132}^{(0,0,\ell,1)}(t,x)-1), 
\end{eqnarray*}

\vspace{-0.5cm}

\begin{eqnarray*}
Q_{132}^{(0,1,\ell,2)}(t,x) &=& 1 + t +
\frac{2t^2}{1-t} + \frac{t}{1-t}Q_{132}^{(0,1,\ell,0)}(t,x)
(Q_{132}^{(0,0,\ell,2)}(t,x)-(1+t)) + \nonumber \\
&&\frac{t^2}{1-t}(Q_{132}^{(0,1,\ell,1)}(t,x) -1), 
\end{eqnarray*}
and 

\begin{eqnarray*}
Q_{132}^{(0,2,\ell,2)}(t,x) &=& 1 + t +2t^2+ \frac{5t^3}{1-t} + 
\frac{t}{1-t}(Q_{132}^{(0,1,\ell,2)}(t,x)-(1+t+2t^2))+ 
\nonumber \\
&&\frac{t}{1-t}(Q_{132}^{(0,2,\ell,0)}(t,x)-1)
(Q_{132}^{(0,0,\ell,2)}(t,x)-(1+t)) + \nonumber \\
&&\frac{t^2}{1-t}(Q_{132}^{(0,2,\ell,1)}(t,x) -(1+t)). 
\end{eqnarray*}

We used these formulas to compute the following. 

\begin{align*}
&Q_{132}^{(0,1,1,1)}(t,x) = 1+t+2 t^2+5 t^3+ (13+x)t^4 +
\left(33+8 x+x^2\right)t^5+
\ \ \ \ \ \ \ \ \ \ \ \ \ \ \ \ \ \ \ \ \ \ \ \ \ \\ 
&  \left(81+39 x+11 x^2+x^3\right)t^6+\left(193+150 x+70 x^2+15 x^3+x^4\right)t^7+\\
&\left(449+501 x+337 x^2+122 x^3+20 x^4+x^5\right)t^8+\\
& \left(1025+1524 x+1363 x^2+719 x^3+204 x^4+26 x^5+x^6\right)t^9+ \\
&\left(2305+4339 x+4891
x^2+3450 x^3+1450 x^4+327 x^5+33 x^6+x^7\right)t^{10}+ \cdots. 
\end{align*}

\vspace{-0.5cm}

\begin{align*}
&Q_{132}^{(0,1,2,1)}(t,x) = 1+t+2 t^2+5 t^3+14 t^4+2  (20+x)t^5+ \left(113+17 x+2 x^2\right)t^6+ \ \ \ \ \ \ \ \ \ \ \ \\
&  \left(314+92 x+21 x^2+2 x^3\right)t^7+\left(859+404 x+140 x^2+25
x^3+2 x^4\right)t^8+ \\
&\left(2319+1567 x+745 x^2+200 x^3+29 x^4+2 x^5\right)
t^9+\\
& \left(6192+5597 x+3438 x^2+1262 x^3+272 x^4+33 x^5+2 x^6\right)t^{10}+ 
\cdots.
\end{align*}

\vspace{-0.5cm}

\begin{align*}
&Q_{132}^{(0,1,3,1)}(t,x) = 1+t+2 t^2+5 t^3+14 t^4+42 t^5+ (127+5 x)t^6+\left(380+44 x+5 x^2\right)t^7+ \\
& 
\left(1125+246 x+54 x^2+5 x^3\right)t^8+\left(3299+1135
x+359 x^2+64 x^3+5 x^4\right)t^9 +\\
& \left(9592+4691 x+1942 x^2+492 x^3+74 x^4+5 x^5\right)t^{10}+\cdots. 
\end{align*}

Our next theorem will explain the coefficient of the highest 
and second highest powers of $x$ that appear in 
$Q_{n,132}^{(0,1,\ell,1)}(x)$ in these series. 

\begin{theorem} \ 
\begin{itemize}

\item[(i)] For $n \geq 3 +\ell$, the highest power of 
$x$ that occurs in $Q_{n,132}^{(0,1,\ell,1)}(x)$ is $x^{n-2-\ell}$ which 
appears with a coefficient of $C_\ell$. 

\item[(ii)] For $n \geq 5$, $Q_{n,132}^{(0,1,1,1)}(x)|_{x^{n-4}} = 
5 +\binom{n-2}{2}$.

\item[(iii)] For all  $\ell \geq 2$ and 
$n \geq 4+\ell$, $Q_{n,132}^{(0,1,2,1)}(x)|_{x^{n-3-\ell}} = 
C_{\ell+1} +6C_\ell +2C_\ell (n-4 -\ell)$.

\end{itemize}
\end{theorem}

\begin{proof}
It is easy to see that the maximum 
number of matches of $\MMP(0,1,\ell,1)$ that are possible in 
a 132-avoiding permutation is a permutation of 
the form $n~\alpha~(n-1)~\beta$ where $\alpha$ is a 132-avoiding 
permutation on the elements $n-\ell-1, \ldots, n-2$ and 
$\beta$ is the decreasing permutation on the elements 
$1, \ldots, n-\ell -2$. Thus, the highest power in $Q_n^{(0,1,\ell,1)}$ is $x^{n-\ell-2}$ which has a coefficient 
of $C_\ell$.

For parts (ii) and  (iii), we note that it follows from 
(\ref{0klm-rec}) that 
\begin{equation*}\label{01ell1rec}
Q_{n,132}^{(0,1,\ell,1)}(x) = Q_{n-1,132}^{(0,1,\ell,1)}(x) +
\sum_{i=1}^{n-1} Q_{i-1,132}^{(0,1,\ell,0)}(x) Q_{n-i,132}^{(0,0,\ell,1)}(x).
\end{equation*}
We proved in \cite{kitremtieII} that  the highest power of $x$ that appears 
in $Q_{n,132}^{(0,1,\ell,0)}(x) =Q_{n,132}^{(0,0,\ell,1)}(x)$ is 
$x^{n-\ell -1}$ which appears with a coefficient of $C_{\ell}$ for 
$n \geq \ell+2$.  This implies that the highest power of $x$ that 
appears in 
$Q_{i-1,132}^{(0,1,\ell,0)}(x) Q_{n-i,132}^{(0,0,\ell,1)}(x)$ is less than 
$x^{n-\ell -3}$ for $i =3, \ldots, n-2$. 

Hence we have four cases to consider when we 
are computing $Q_{n,132}^{(0,1,1,1)}(x)|_{x^{n-4}}$. \\
\ \\
{\bf Case 1.} $i=1$. In this case, $Q_{i-1,132}^{(0,1,1,0)}(x)
Q_{n-i,132}^{(0,0,1,1)}(x)|_{x^{n-4}} = Q_{n-1,132}^{(0,0,1,1)}(x)|_{x^{n-4}}$ 
and we proved in \cite{kitremtieII} that 
$$Q_{n-1,132}^{(0,0,1,1)}(x)|_{x^{n-4}} = 
Q_{n-1,132}^{(0,1,1,0)}(x)|_{x^{n-4}} = 
2+\binom{n-2}{2} \ \mbox{for } n \geq 5.$$
{\bf Case 2.} $i=2$. In this case, $Q_{i-1,132}^{(0,1,1,0)}(x)
Q_{n-i,132}^{(0,0,1,1)}(x)|_{x^{n-4}} = Q_{n-2,132}^{(0,0,1,1)}(x)|_{x^{n-4}}$ 
and we proved in \cite{kitremtieII} that 
$$Q_{n-2,132}^{(0,0,1,1)}(x)|_{x^{n-4}} = 
Q_{n-2,132}^{(0,1,1,0)}(x)|_{x^{n-4}} = 1 \ \mbox{for } n \geq 5.$$
{\bf Case 3.} $i=n-1$. In this case, $Q_{i-1,132}^{(0,1,1,0)}(x)
Q_{n-i,132}^{(0,0,1,1)}(x)|_{x^{n-4}} = Q_{n-2,132}^{(0,1,1,0)}(x)|_{x^{n-4}}$ 
and we proved in \cite{kitremtieII} that 
$$Q_{n-2,132}^{(0,1,1,0)}(x)|_{x^{n-4}}= 1 \ \mbox{for } n \geq 5.$$
{\bf Case 4.} $Q_{n-1,132}^{(0,1,1,1)}(x)|_{x^{n-4}}$. By part (i), we know that $Q_{n-1,132}^{(0,1,1,1)}(x)|_{x^{n-4}} = 1 \ \mbox{for } n \geq 5$. \\

Thus, $Q_{n,132}^{(0,1,1,1)}(x)|_{x^{n-4}} =5+\binom{n-2}{2}$ for $n \geq 5$.

Again there are four cases to consider when computing 
$Q_{n-1,132}^{(0,1,\ell,1)}(x)|_{x^{n-3-\ell}}$ for $\ell \geq 2$.\\

{\bf Case 1.} $i=1$. In this case, $Q_{i-1,132}^{(0,1,\ell,0)}(x)
Q_{n-i,132}^{(0,0,\ell,1)}(x)|_{x^{n-3 -\ell}} = 
Q_{n-1,132}^{(0,0,\ell,1)}(x)|_{x^{n-3-\ell}}$ 
and we proved in \cite{kitremtieII} that  
\begin{eqnarray*}
Q_{n-1,132}^{(0,0,\ell,1)}(x)|_{x^{n-3-\ell}} &=&
Q_{n-1,132}^{(0,1,\ell,0)}(x)|_{x^{n-n-3 -\ell}} \\
&=&  C_{\ell+1}+C_\ell +2C_{\ell}(n-3 -\ell) \\
&=& C_{\ell+1}+3C_\ell +2C_{\ell}(n-4 -\ell) \ \mbox{for } n \geq 4+\ell. 
\end{eqnarray*}
\ \\
{\bf Case 2.} $i=2$. In this case, $Q_{i-1,132}^{(0,1,\ell,0)}(x)
Q_{n-i,132}^{(0,0,\ell,1)}(x)|_{x^{n-3-\ell}} = 
Q_{n-2,132}^{(0,0,\ell,1)}(x)|_{x^{n-3-\ell}}$ 
and we proved in \cite{kitremtieII} that 
$$Q_{n-2,132}^{(0,0,\ell,1)}(x)|_{x^{n-3 -\ell}} = 
Q_{n-2,132}^{(0,1,\ell,0)}(x)|_{x^{n-3-\ell}}= C_\ell \ 
\mbox{for } n \geq 4+\ell.$$
{\bf Case 3.} $i=n-1$. In this case, $Q_{i-1,132}^{(0,1,\ell,0)}(x)
Q_{n-i,132}^{(0,0,\ell,1)}(x)|_{x^{n-3-\ell}} = 
Q_{n-2,132}^{(0,1,\ell,0)}(x)|_{x^{n-3-\ell}}$ 
and we proved in \cite{kitremtieII} that 
$$Q_{n-2,132}^{(0,1,\ell,0)}(x)|_{x^{n-3-\ell}} =  C_\ell \ \mbox{for } 
n \geq 4+\ell.$$
{\bf Case 4.} $Q_{n-1,132}^{(0,1,\ell,1)}(x)|_{x^{n-3-\ell}}$. By part (i), we know that 
$$Q_{n-1,132}^{(0,1,\ell,1)}(x)|_{x^{n-3-\ell}} = C_\ell \ \mbox{for } 
n \geq 4+\ell.$$

Thus, $$Q_{n,132}^{(0,1,\ell,1)}(x)|_{x^{n-3-\ell}} =C_{\ell+1} +6C_{\ell} 
+2C_{\ell}(n-4-\ell) \mbox{ for }  n\geq 4+\ell.$$

For example, when $\ell =2$, we have  that 
$$Q_{n,132}^{(0,1,2,1)}(x)|_{x^{n-5}} =17 +4(n-6) \ \mbox{for } n \geq 6$$ 
and, for $\ell = 3$, we have that 
$$Q_{n,132}^{(0,1,3,1)}(x)|_{x^{n-5}} =44 +10(n-7) \ \mbox{for } n \geq 7$$ 
which agrees with the series we computed. 
\end{proof}

The sequence 
$(Q_{n,132}^{(0,1,1,1)}(0))_{n \geq 1}$ starts out 
$1,2,5,13,33,81,193,449, \ldots $. This is the sequence A005183 
in OEIS.  Using the fact that 
$Q_{132}^{(0,1,1,0)}(t,0) = Q_{132}^{(0,0,1,1)}(t,0) = \frac{1-t}{1-2t}$, one 
can show that 
$$Q_{132}^{(0,1,1,1)}(t,x)= \frac{1-4t+5t^2-t^3}{(1-2t)^2(1-t)}$$
from which it is possible to show that 
$Q_{n,132}^{(0,1,1,1)}(0) = (n-1)2^{n-2}+1$ for $n \geq 1$.

The sequence 
$(Q_{n,132}^{(0,1,1,1)}(x)|_x)_{n \geq 4}$ starts out 
$1,8,39,150,501,1524 \ldots $. This seems to be the sequence A055281 in the OEIS. 
The $n$-th term of this sequence $(n^2-n+4)2^{n+1}-7-n$ and is 
the number of {\em directed column-convex polyominoes} of area $n+5$ having 
along the lower contour exactly 2 {\em reentrant corners}. 

\begin{problem} Verify that the sequence 
$(Q_{n,132}^{(0,1,1,1)}(x)|_x)_{n \geq 4}$ is counted by $$(n^2-9n+24)2^{n-3}-3-n$$ and if so, find a bijective correspondence with the polyominoes described in A055281 in the OEIS. \end{problem}

We have computed that 
\begin{align*}
&Q_{132}^{(0,1,1,2)}(t,x) =  1+t+2 t^2+5 t^3+14 t^4+2  (20+x)t^5+ 
\left(111+19 x+2 x^2\right)t^6 +\\
& \left(296+106 x+25 x^2+2 x^3\right)t^7+ \left(761+456 x+178 x^2+33
x^3+2 x^4\right)t^8+\\
& \left(1898+1677 x+947 x^2+295 x^3+43 x^4+2 x^5\right)t^9+\\
& \left(4619+5553 x+4191 x^2+1901 x^3+475 x^4+55 x^5+2 x^6\right)t^{10} +
\cdots.
\end{align*}

\vspace{-0.5cm}

 \begin{align*}
&Q_{132}^{(0,1,2,2)}(t,x) = 1+t+2 t^2+5 t^3+14 t^4+42 t^5+4  (32+x)t^6+
\ \ \ \ \ \ \ \ \ \ \ \ \ \ \ \ \ \ \ \ \ \\
& \left(385+40 x+4 x^2\right)t^7+ \left(1135+243 x+48 x^2+4 x^3\right)t^8+\\
& \left(3281+1170
x+351 x^2+56 x^3+4 x^4\right)t^9+\\
& \left(9324+4905 x+2016 x^2+483 x^3+64 x^4+4 x^5\right)t^{10}+\cdots.
\end{align*}

\vspace{-0.5cm}

 \begin{align*}
&Q_{132}^{(0,1,3,2)}(t,x) =
1+t+2 t^2+5 t^3+14 t^4+42 t^5+132 t^6+(419+10 x)t^7+ \ \ \ \ \ \ \\
&  \left(1317+103 x+10 x^2\right)t^8+\left(4085+644 x+123 x^2+10 x^3\right)t^9 +\\
&
\left(12514+3229 x+900 x^2+143 x^3+10 x^4\right)t^{10}+ \\
&\left(37913+14282 x+5222 x^2+1196 x^3+163 x^4+10 x^5\right)t^{11}+ \cdots.
\end{align*}

Again we can explain the coefficients of the highest and second 
highest coefficients in $Q_{n,132}^{(0,1,\ell,2)}(x)$ for large enough
$n$.  

\begin{theorem} \ 

\begin{itemize}

\item[(i)] For $n \geq \ell +4$, the highest power of $x$ that 
appears in $Q_{n,132}^{(0,1,\ell,2)}(x)$ is $x^{n-3-\ell}$ which occurs 
with a coefficient of $2C_\ell$. 

\item[(ii)] For $n \geq 7$, $\displaystyle Q_{n,132}^{(0,1,1,2)}(x)|_{x^{n-5}}= 13+  \binom{n-2}{2}$. 

\item[(iii)] For all $\ell \geq 2$ and $n \geq 5 + \ell$, 
$\displaystyle Q_{n,132}^{(0,1,\ell,2)}(x)|_{x^{n-4-\ell}}= 2C_{\ell+1} +15C_{\ell} +4C_{\ell}(n-5-\ell)$. 

\end{itemize}
\end{theorem}

\begin{proof}

For (i), it is easy to see that the maximum 
number of matches of $\MMP(0,1,\ell,1)$ that are possible in 
a 132-avoiding permutation is a permutation of 
the form $n~\alpha~(n-1)~\beta$ where $\alpha$ is a 132-avoiding 
permutation on the elements $n-\ell-1, \ldots, n-2$ and 
$\beta = (n-\ell-2) (n-\ell -3) \ldots 321$ or 
$\beta = (n-\ell-2) (n-\ell -3) \ldots 312$.
Thus, the highest power 
in $Q_{n,132}^{(0,1,\ell,2)}(x)$ is $x^{n-\ell-3}$ which has a coefficient of 
$2C_\ell$.

For (ii) and (iii), we note that the recursion for $Q_{n,132}^{(0,1,\ell,2)}(x)$ 
is 
$$Q_{n,132}^{(0,1,\ell,2)}(x) = Q_{n-1,132}^{(0,1,\ell,2)}(x) +
Q_{n-2,132}^{(0,1,\ell,1)}(x) +\sum_{i=1}^{n-2} 
Q_{i-1,132}^{(0,1,\ell,0)}(x)Q_{n-i,132}^{(0,0,\ell,2)}(x).$$

Since the highest power of $x$ that occurs in 
$Q_{n,132}^{(0,1,\ell,0)}(x)$ is $x^{n-1-\ell}$ and the highest power 
of $x$ that occurs in $Q_{n,132}^{(0,0,\ell,2)}(x)$ is 
$n-2-\ell$, it follows that the highest power of 
$x$ that occurs in 
$Q_{i-1,132}^{(0,1,\ell,0)}(x)Q_{n-i,132}^{(0,0,\ell,2)}(x)$ is less 
than $x^{n-4-\ell}$ for $i=4, \dots ,n-3$. 

Thus, we have to consider five cases when computing 
$Q_{n,132}^{(0,1,\ell,2)}(x)|_{x^{n-4-\ell}}$. \\
\ \\
{\bf Case 1.} $Q_{n-1,132}^{(0,1,\ell,2)}(x)|_{x^{n-4-\ell}}$. By part (i), 
$$Q_{n-1,132}^{(0,1,\ell,2)}(x)|_{x^{n-4-\ell}}=2C_{\ell} \ \mbox{for } 
n \geq \ell+5.$$
{\bf Case 2.} $Q_{n-2,132}^{(0,1,\ell,1)}(x)|_{x^{n-4-\ell}}$. We have shown 
earlier that 
$$Q_{n-1,132}^{(0,1,\ell,1)}(x)|_{x^{n-4-\ell}}=C_{\ell} \ \mbox{for } 
n \geq \ell+5.$$
\ \\
{\bf Case 3.} $i=n-2$. In this case, $Q_{i-1,132}^{(0,1,\ell,0)}(x)Q_{n-i,132}^{(0,0,\ell,2)}(x)$ 
equals $2Q_{n-3,132}^{(0,1,\ell,0)}(x)$. We have shown in \cite{kitremtieII} 
that $Q_{n-3,132}^{(0,1,\ell,0)}(x)|_{x^{n-4-\ell}} = C_\ell$ for 
$n \geq \ell +5$ so 
that we get a contribution of $2C_{\ell}$ in this case. \\
\ \\
{\bf Case 4.} $i=2$. In this case, $Q_{i-1,132}^{(0,1,\ell,0)}(x)Q_{n-i,132}^{(0,0,\ell,2)}(x)$ 
equals $Q_{n-2,132}^{(0,0,\ell,2)}(x)$. We have shown in \cite{kitremtieII} 
that 
$$Q_{n-3,132}^{(0,0,\ell,2)}(x)|_{x^{n-4-\ell}} = 2C_\ell \ \mbox{for } 
n \geq \ell+5.$$ 
{\bf Case 5.} $i=1$. In this case, $Q_{i-1,132}^{(0,1,\ell,0)}(x)Q_{n-i,132}^{(0,0,\ell,2)}(x)$ 
equals $Q_{n-1,132}^{(0,0,\ell,2)}(x)$. We have shown in \cite{kitremtieII}  that for $n \geq \ell +5$,
$$Q_{n-3,132}^{(0,0,\ell,2)}(x)|_{x^{n-4-\ell}} = 
\begin{cases} 6+2\binom{n-2}{2} & \mbox{if $\ell =1$}, \\
2C_{\ell+1} +8C_\ell +4C_{\ell}(n-5-\ell) & \mbox{if $\ell \geq 2$}.
\end{cases}
$$
Thus, for $\ell =1$, we get 
$$Q_{n-i,132}^{(0,1,1,2)}(x)|_{x^{n-5}} = 13 +2 \binom{n-2}{2} \ \mbox{for } 
n \geq 6$$
and, for $\ell \geq 2$, 
$$Q_{n-i,132}^{(0,1,\ell,2)}(x)|_{x^{n-4-\ell}} = 2C_{\ell+1} +15C_{\ell} + 
4C_{\ell}(n-5-\ell)  \ \mbox{for } 
n \geq 5+\ell.$$
For example, when $\ell =2$, we get 
$$Q_{n-i,132}^{(0,1,2,2)}(x)|_{x^{n-6}} = 40+8(n-7) \ \mbox{for } 
n \geq 7$$
and, for $\ell =3$, we get 
$$Q_{n-i,132}^{(0,1,2,2)}(x)|_{x^{n-6}} = 103+20(n-8) \ \mbox{for } 
n \geq 8$$
which agrees with the series that we computed. \end{proof}

\begin{align*}
&Q_{132}^{(0,2,1,2)}(t,x) =  1+t+2 t^2+5 t^3+14 t^4+42 t^5+4 t^6 (32+x)+ \left(380+45 x+4 x^2\right)t^7+\\
& \left(1083+286 x+57 x^2+4 x^3\right)t^8+ \left(2964+1368
x+453 x^2+73 x^3+4 x^4\right)t^9+\\
& \left(7831+5501 x+2650 x^2+717 x^3+93 x^4+4 x^5\right)t^{10}+\\
&\left(20092+19675 x+12749 x^2+5035 x^3+1114 x^4+117
x^5+4 x^6\right)t^{11}+ \cdots.
\end{align*}

\vspace{-0.5cm}

\begin{align*}
&Q_{132}^{(0,2,2,2)}(t,x) =1+t+2 t^2+5 t^3+14 t^4+42 t^5+132 t^6+ 
(421+8 x)t^7+\ \ \ \ \ \ \ \ \ \ \ \ \ \ \ \ \ \\
&\left(1328+94x+4 x^2\right)t^8+ \left(4103+641 x+110 x^2+8 x^3\right)
t^9+\\
&
\left(12401+3376 x+885 x^2+126 x^3+8 x^4\right)t^{10}+\\
& \left(36740+15235 x+5484 x^2+1177 x^3+142 x^4+8 x^5\right)t^{11}+\\
& \left(106993+62012 x+28872
x^2+8452 x^3+1517 x^4+158 x^5+8 x^6\right)t^{12}+ \cdots. 
\end{align*}

\vspace{-0.5cm}

\begin{align*}
&Q_{132}^{(0,2,3,2)}(t,x) = 1+t+2 t^2+5 t^3+14 t^4+42 t^5+132 t^6+
429 t^7+\ \ \ \ \ \ \ \ \ \ \ \ \ \ \ \ \ \ \ \ \ \ \ \ \ \ \\
&(1410+20 x)t^8+ \left(4601+241 x+20 x^2\right)t^9+ \left(14809+1686 x+281 x^2+20
x^3\right)t^{10}+\\
& \left(46990+9187 x+2268 x^2+321 x^3+20 x^4\right)t^{11}+\\
& \left(147163+43394 x+14144 x^2+2930 x^3+361 x^4+20 x^5\right)t^{12}+ \cdots.
\end{align*}

It is easy to explain the coefficient of the highest power in 
$Q_n^{(0,2,\ell,2)}(x)$. That is, the maximum 
number of matches of $\MMP(0,2,\ell,2)$ that are possible in 
a 132-avoiding permutation is a permutation of 
the form $n(n-1)~\alpha~(n-2)~\beta$ or  
$(n-1)n~\alpha~(n-2)~\beta$
where $\alpha$ is a 132-avoiding 
permutation on the elements $n-\ell-2, \ldots, n-3$ and 
$\beta = (n-\ell-3) (n-\ell -4) \ldots 321$ or 
$\beta = (n-\ell-3) (n-\ell -4) \ldots 312$.
Thus, the highest power 
in $Q_n^{(0,1,\ell,2)}$ is $x^{n-\ell-5}$ which has a coefficient of 
$4C_\ell$.

\section{$Q_{n,132}^{(\ell,k,0,m)}(x) = Q_{n,132}^{(\ell,m,0,k)}(x)$ 
where $k,\ell,m \geq 1$}

By Lemma \ref{sym},  we need only consider $Q_{n,132}^{(\ell,k,0,m)}(x)$. Suppose that $k, \ell, m \geq 1$ and $n \geq k+m$.  
It is clear that $n$ can never match 
$\MMP(\ell,k,0,m)$ for $k,\ell, m \geq 1$ in any 
$\sg \in S_n(132)$. If $\sg = \sg_1 \ldots \sg_n \in S_n(132)$  
and $\sg_i =n$, then we have three cases, depending 
on the value of $i$. \\
 \\
{\bf Case 1.} $i < k$. It is easy to see that as we sum 
over all the permutations $\sg$ in $S_n^{(i)}(132)$, our choices 
for the structure for $A_i(\sg)$ will contribute a factor 
of $C_{i-1}$ to $Q_{n,132}^{(\ell,k,0,m)}(x)$ since 
the elements in $A_i(\sg)$ do not have enough elements 
to the left to match $\MMP(\ell,k,0,m)$ in $\sg$. 
Similarly, our choices 
for the structure for $B_i(\sg)$ will contribute a factor 
of $Q_{n-i,132}^{(\ell,k-i,0,m)}(x)$ to $Q_{n,132}^{(\ell,k,0,m)}(x)$ 
since $\sg_1 \ldots \sg_i$ will automatically be 
in the second quadrant relative to the coordinate system 
with the origin at $(s,\sg_s)$ for any $s > i$.  Thus, the 
permutations in Case~1 will contribute 
$$\sum_{i=1}^{k-1} C_{i-1}
Q_{n-i,132}^{(\ell,k-i,0,m)}(x)$$ 
to $Q_{n,132}^{(\ell,k,0,m)}(x)$.\\
\ \\
{\bf Case 2.} $k \leq i < n-m$. It is easy to see that as we sum 
over all the permutations $\sg$ in $S_n^{(i)}(132)$, our choices 
for the structure for $A_i(\sg)$ will contribute a factor 
of $Q_{i-1,132}^{(\ell-1,k,0,0)}(x)$ to $Q_{n,132}^{(\ell,k,0,m)}(x)$ since 
the elements in $B_i(\sg)$ will all be in the fourth quadrant  
and $\sg_i =n$ is in the first quadrant
relative to a coordinate system centered at $(r,\sg_r)$ for 
$r \leq i$  in this case. 
Similarly, our choices 
for the structure for $B_i(\sg)$ will contribute a factor 
of $Q_{n-i,132}^{(\ell,0,0,m)}(x)$ to $Q_{n,132}^{(\ell,k,0,m)}(x)$ 
since $\sg_1 \ldots \sg_i$ will automatically be 
in the second quadrant relative to the coordinate system 
with the origin at $(s,\sg_s)$ for any $s > i$.  Thus, the 
permutations in Case~2 will contribute 
$$\sum_{i=k}^{n-m} Q_{i-1,132}^{(\ell-1,k,0,0)}(x) 
Q_{n-i,132}^{(\ell,0,0,m)}(x)$$ 
to $Q_{n,132}^{(\ell,k,0,m)}(x)$.\\
\ \\
{\bf Case 3.} $i \geq n-m +1$. It is easy to see that as we sum 
over all the permutations $\sg$ in $S_n^{(i)}(132)$, our choices 
for the structure for $A_i(\sg)$ will contribute a factor 
of $Q_{i-1,132}^{(\ell-1,k,0,m-(n-i))}(x)$ 
to $Q_{n,132}^{(\ell,k,0,m)}(x)$ since $\sg_i =n$ 
will be in the first quadrant and 
the elements in $B_i(\sg)$ will all be in the fourth quadrant 
relative to a coordinate system centered at $(r,\sg_r)$ for 
$r \leq i$ in this case. 
Similarly, our choices 
for the structure for $B_i(\sg)$ will contribute a factor 
of $C_{n-i}$ to $Q_{n,132}^{(\ell,0,0,m)}(x)$ 
since $\sg_j$ where $j >i$ does not have enough elements 
to its right to match $\MMP(\ell,k,0,m)$ in $\sg$.  Thus, the 
permutations in Case 3 will contribute 
$$\sum_{i=n-m+1}^{n} Q_{i-1,132}^{(\ell-1,k,0,m-(n-i))}(x) 
C_{n-i}$$ 
to $Q_{n,132}^{(\ell,k,0,m)}(x)$. Thus, 
we have the following. For $ n \geq k + m $,  
\begin{eqnarray}\label{lk0m-rec} 
Q_{n,132}^{(\ell,k,0,m)}(x) &=& \sum_{i=1}^{k-1} C_{i-1} 
Q_{n-i,132}^{(\ell,k-i,0,m)}(x) + 
\sum_{i=k}^{n-m} Q_{i-1,132}^{(\ell-1,k,0,0)}(x) 
Q_{n-i,132}^{(\ell,0,0,m)}(x) +\nonumber \\
&& \sum_{i=n-m+1}^{n} Q_{i-1,132}^{(\ell-1,k,0,m-(n-i))}(x) C_{n-i}.
\end{eqnarray}
Multiplying (\ref{lk0m-rec}) by $t^n$ and summing over $n$ will yield the 
following theorem. 

\begin{theorem}\label{thm:Qlk0m}
 For all $\ell,k,m \geq 1$, 
\begin{eqnarray}\label{eq:Qlk0m0}
Q_{132}^{(\ell,k,0,m)}(t,x) &=& \sum_{p=0}^{k+m-1} C_p t^p + \nonumber \\
&&t\sum_{i=0}^{k-2} C_i t^i \left(  Q_{132}^{(\ell,k-1-i,0,m)}(t,x) -
\sum_{r=0}^{k-i+m-2} C_rt^r \right) + \nonumber \\
&&t \left( Q_{132}^{(\ell-1,k,0,0)}(t,x) - \sum_{a=0}^{k-2} C_a t^a \right) 
\left( Q_{132}^{(\ell,0,0,m)}(t,x) - \sum_{b=0}^{m-1} C_b t^b \right) + 
\nonumber \\
&& t\sum_{j=0}^{m-1} C_j t^j \left( Q_{132}^{(\ell-1,k,0,m-j)}(t,x) - 
\sum_{s=0}^{k+m-j-2} C_st^s\right).
\end{eqnarray}
\end{theorem}

Note that we can compute $Q_{132}^{(\ell,k,0,0)}(t,x) =Q_{132}^{(\ell,0,0,k)}(t,x)$ by Theorem \ref{thm:Q0kl0} so that (\ref{eq:Qlk0m0}) allows 
us to compute $Q_{132}^{(\ell,k,0,m)}(t,x)$ for any 
$k,\ell, m \geq 0$.

\subsection{Explicit formulas for  $Q^{(\ell,k,0,m)}_{n,132}(x)|_{x^r}$}

It follows from Theorem \ref{thm:Qlk0m} that  
\begin{equation}\label{eq:1101}
Q_{132}^{(\ell,1,0,1)}(t,x) = 1+t+  tQ_{132}^{(\ell-1,1,0,0)}(t,x)
\left( Q_{132}^{(\ell,0,0,1)}(t,x)-1\right)+t\left( Q_{132}^{(\ell-1,1,0,1)}(t,x)-1\right),
\end{equation}

\vspace{-0.5cm}

\begin{eqnarray*}\label{eq:1102}
Q_{132}^{(\ell,1,0,2)}(t,x) &=& 1 + t + 2t^2 + tQ_{132}^{(\ell-1,1,0,0)}(t,x)
\left( Q_{132}^{(\ell,0,0,2)}(t,x)-(1+t)\right)+
\ \ \ \ \ \ \ \ \ \ \ \ \ \nonumber \\
&&t\left( Q_{132}^{(\ell-1,1,0,2)}(t,x)-(1+t) + 
t\left(Q_{132}^{(\ell-1,1,0,1)}(t,x)-1\right)\right), 
\end{eqnarray*}
and 

\begin{eqnarray*}\label{eq:1202}
Q_{132}^{(\ell,2,0,2)}(t,x) &=& 1 + t +2t^2+ 5t^3 + 
t\left(Q_{132}^{(\ell,1,0,2)}(t,x)-(1+t+2t^2)  \right)+ \nonumber \\
&&t\left( Q_{132}^{(\ell-1,2,0,0)}(t,x)-1\right)
\left( Q_{132}^{(\ell,0,0,2)}(t,x)-(1+t)\right)+\nonumber \\
&&t\left( Q_{132}^{(\ell-1,2,0,2)}(t,x)-(1+t+2t^2) + 
t\left(Q_{132}^{(\ell-1,2,0,1)}(t,x)-(1+t)\right) \right). 
\end{eqnarray*}

One can use these formulas to compute the following.

\begin{align*}
&Q_{132}^{(1,1,0,1)}(t,x) =  1+t+2 t^2+5 t^3+2  (5+2 x)t^4+ 
\left(17+17 x+8 x^2\right)t^5+\\
& \left(26+44 x+42 x^2+20 x^3\right)t^6+\left(37+90 x+129 x^2+117 x^3+56
x^4\right) t^7 +\\
& \left(50+160 x+305 x^2+397 x^3+350 x^4+168 x^5\right)t^8+\\
& \left(65+259 x+615 x^2+1029 x^3+1268 x^4+1098 x^5+528 x^6\right)t^9+\\
&
\left(82+392 x+1113 x^2+2259 x^3+3503 x^4+4167 x^5+3564 x^6+1716 x^7\right)t^{10} + \cdots.
\end{align*}

It is easy to explain the highest coefficient of $x$ in 
$Q_{n,132}^{(1,\ell,0,1)}(x)$.  
\begin{theorem} For $n \geq 3+\ell$, the highest power of 
$x$ that occurs in $Q_{n,132}^{(1,\ell,0,1)}(x)$ is 
$x^{n-2 -\ell}$ which occurs with a coefficient of 
$4C_{\ell} C_{n-\ell-2}$.
\end{theorem}
\begin{proof}
It is easy to see that  the maximum number of 
$\MMP(\ell,1,0,1)$ matches occurs in $\sg \in S_n(132)$ when 
$\sg$ is of the form $n~\tau~\alpha~(n-1)$, $n~\tau~(n-1)~\alpha$,  
$(n-1)~\tau~\alpha~n$, or $(n-1)~\tau~n~\alpha$ where $\alpha$ is 
a 132-avoiding permutation on the elements $1,\ldots, \ell$ and 
$\tau$ is a 132-avoiding 
permutations of the elements $\ell +1, \ldots,n-2$. Thus, the highest 
power of $x$ in $Q_{n,132}^{(\ell,1,0,1)}(x)$ for $n \geq \ell+3$ is 
$x^{n-2-\ell}$ which occurs with a coefficient of $4C_{\ell} C_{n-\ell-2}$. 
\end{proof}

We can also explain the second highest coefficient in 
$Q_{n,132}^{(1,1,0,1)}(x)$.
\begin{theorem} For $n \geq 5$, 
$$Q_{n,132}^{(1,1,0,1)}(x)|_{x^{n-4}} = 8C_{n-3} +C_{n-4}.$$
\end{theorem} 
\begin{proof}
In this case, the recursion for $Q_{n,132}^{(1,1,0,1)}(x)$ is 
\begin{equation*}
Q_{n,132}^{(1,1,0,1)}(x) = Q_{n-1,132}^{(0,1,0,1)}(x) + 
\sum_{i=1}^{n-1} Q_{i-1,132}^{(0,1,0,0)}(x)Q_{n-i,132}^{(1,0,0,1)}(x).
\end{equation*}
It was proved in \cite{kitremtie} that for $n \geq 1$, the  highest power of $x$ that occurs in $Q_{n,132}^{(0,1,0,0)}(x)$ is $x^{n-1}$ which occurs with a coefficient of $C_{n-1}$. 
It was proved in \cite{kitremtieII} that for $n \geq 3$, the highest power of $x$ that occurs in $Q_{n,132}^{(1,0,0,1)}(x)$ is $x^{n-2}$ which occurs with a coefficient of $2C_{n-2}$. 
It follows that 
\begin{eqnarray*}
Q_{n,132}^{(1,1,0,1)}(x)|_{x^{n-4}} &=& 
Q_{n-1,132}^{(0,1,0,1)}(x)|_{x^{n-4}} + 
Q_{n-1,132}^{(1,0,0,1)}(x)|_{x^{n-4}} + 
Q_{n-2,132}^{(0,1,0,0)}(x)|_{x^{n-4}} + \\
&& \sum_{i=2}^{n-2} Q_{i-1,132}^{(0,1,0,0)}(x)|_{x^{i-2}}
Q_{n-i,132}^{(1,0,0,1)}(x)|_{x^{n-i-2}}.
\end{eqnarray*}
It was shown in \cite{kitremtie} and \cite{kitremtieII} that   
\begin{eqnarray*} 
Q_{n-1,132}^{(0,1,0,1)}(x)|_{x^{n-4}} &=& 2C_{n-3}+C_{n-4} 
\ \mbox{for } n \geq 5,\\
Q_{n-1,132}^{(1,0,0,1)}(x)|_{x^{n-4}} &=& 3C_{n-3}  
\ \mbox{for } n \geq 5, \ \mbox{and} \\
Q_{n-2,132}^{(0,1,0,0)}(x)|_{x^{n-4}} &=& C_{n-3} 
\ \mbox{for } n \geq 5.
\end{eqnarray*}
Thus, for $n \geq 5$, 
\begin{eqnarray*}
Q_{n,132}^{(1,1,0,1)}(x)|_{x^{n-4}} &=& 2C_{n-3}+C_{n-4} + 3C_{n-3} +C_{n-3} +  \sum_{i=2}^{n-2}C_{i-2}2C_{n-i-2} \\
&=& 6 C_{n-3} +C_{n-4} + 2 \sum_{i=2}^{n-2}C_{i-2}C_{n-i-2} \\
&=&6 C_{n-3} +C_{n-4}+2C_{n-3} = 8C_{n-3}+C_{n-4}.
\end{eqnarray*}
\end{proof}

The sequence $(Q_{n,132}^{(1,1,0,1)}(0))_{n \geq 1}$ starts out 
$1,2,5,10,17,26,37,50,82,\ldots$ which 
is the sequence A002522 in the OEIS.  The $n$-th element of the sequence 
has the formula $(n-1)^2+1$. This can be verified by 
computing the generating function $Q_{132}^{(1,1,0,1)}(t,0)$. That is, 
we proved in \cite{kitremtie} and \cite{kitremtieII} that 
\begin{eqnarray*}
Q_{132}^{(0,1,0,0)}(t,0) &=& \frac{1}{1-t},\\
Q_{132}^{(1,0,0,1)}(t,0) &=& \frac{1-2t+2t^2}{(1-t)^3}, \mbox{and} \\
Q_{132}^{(0,1,0,1)}(t,0) &=& \frac{1}{1-t}+\frac{t^2}{(1-t)^2}.
\end{eqnarray*}
Plugging these formulas into (\ref{eq:1101}), one can compute 
that 
$$Q_{132}^{(1,1,0,1)}(t,0) = \frac{1-3t+5t^2-2t^3+t^4}{(1-t)^3}.$$

\begin{problem} Find a direct combinatorial proof of the fact that
$Q_{n,132}^{(1,1,0,1)}(0) = (n-1)^2+1$ for $n \geq 1$. \end{problem}
\begin{align*}
&Q_{132}^{(2,1,0,1)}(t,x) = 1+t+2 t^2+5 t^3+14 t^4+ (33+9 x)t^5+ 
\left(71+43 x+18 x^2\right)t^6+\\
& \left(146+137 x+101 x^2+45 x^3\right) t^7+ \\
&\left(294+368 x+367
x^2+275 x^3+126 x^4\right)t^8+\\
& \left(587+906 x+1100 x^2+1079 x^3+812 x^4+378 x^5\right)t^9+\\
& \left(1169+2125 x+2973 x^2+3463 x^3+3352 x^4+2526
x^5+1188 x^6\right)t^{10}+\cdots
\end{align*}

\vspace{-0.5cm}

\begin{align*}
&Q_{132}^{(3,1,0,1)}(t,x) = 1+t+2 t^2+5 t^3+14 t^4+42 t^5+(116+16x)t^6+\\
& \left(308+89 x+32 x^2\right)t^7+ \left(807+341 x+202 x^2+80 x^3\right)t^8+\\
& \left(2108+1140
x+849 x^2+541 x^3+224 x^4\right)t^9+ \\
&\left(5507+3583 x+3046 x^2+2406 x^3+1582 x^4+672 x^5\right)t^{10}+\\
& \left(14397+10897 x+10141 x^2+9039 x^3+7310
x^4+4890 x^5+2112 x^6\right)t^{11}+\cdots.
\end{align*}

It is not difficult to show that for $n \geq k+3$, 
the highest power or $x$ that occurs in $Q_{n,132}^{(k,1,0,1)}(x)$ is 
$x^{n-k-2}$ which appears with a coefficient of 
$(k+1)^2C_{n-k-2}$.  That is, the maximum number of 
occurrences of $MMP(k,1,0,1)$ for a $\sg \in S_n(132)$ occurs when 
$\sg$ is of the form $x \tau \beta$ where 
$x \in \{n-k, \ldots, n\}$, $\beta$ is a shuffle of 1 with 
the increasing sequence which results from 
$(n-k) (n-k+1) \ldots n$ by removing $x$, and $\tau$ is a 
132-avoiding permutation on $2, \ldots, n-k-1$.  Thus 
we have $k+1$ choices for $x$ and, once $x$ is chosen, we have 
$k+1$ choices for $\beta$, and $C_{n-k-2}$ choices for $\tau$.

%

\begin{align*}
&Q_{132}^{(1,1,0,2)}(t,x) = 1+t+2 t^2+5 t^3+14 t^4+(32+10 x)t^5+
\ \ \ \ \ \ \ \  \ \ \ \ \ \ \ \ \ \ \ \ \ \ \ \ \ \\
& \left(62+50 x+20 x^2\right)t^6+ \left(107+149 x+123 x^2+50 x^3\right)t^7+\\
& \left(170+345 x+433
x^2+342 x^3+140 x^4\right)t^8+\\
& \left(254+685 x+1154 x^2+1327 x^3+1022 x^4+420 x^5\right)t^9+\\
& \left(362+1225 x+2589 x^2+3868 x^3+4228 x^4+3204
x^5+1320 x^6\right)t^{10}+ \cdots .
\end{align*}

\vspace{-0.5cm}

\begin{align*}
&Q_{132}^{(2,1,0,2)}(t,x) = 1+t+2 t^2+5 t^3+14 t^4+42 t^5+(105+27 x)t^6+
\ \ \ \ \ \ \ \ \ \ \ \ \ \ \ \ \\
& \left(235+140 x+54 x^2\right)t^7+ \left(494+470 x+331 x^2+135 x^3\right)
t^8+\\
& \left(1004+1301
x+1275 x^2+904 x^3+378 x^4\right)t^9+\\
& \left(2007+3248 x+3960 x^2+3773 x^3+2674 x^4+1134 x^5\right)t^{10}+ 
\cdots .
\end{align*}

\vspace{-0.5cm}

\begin{align*}
&Q_{132}^{(3,1,0,2)}(t,x) = 1+t+2 t^2+5 t^3+14 t^4+42 t^5+132 t^6+ (373+56 x)
t^7+\ \ \ \\
&\left(998+320 x+112 x^2\right) t^8+ \left(2615+1233 x+734 x^2+280 x^3\right)t^9+\\
&
\left(6813+4092 x+3131 x^2+1976 x^3+784 x^4\right)t^{10}+\\
& \left(17749+12699 x+11223 x^2+8967 x^3+5796 x^4+2352 x^5\right)t^{11}+ 
\cdots .
\end{align*}

\vspace{-0.5cm}

\begin{align*}
&Q_{132}^{(1,2,0,2)}(t,x) = 1+t+2 t^2+5 t^3+14 t^4+42 t^5+ (107+25 x)t^6+
\ \ \ \ \ \ \ \ \ \ \ \ \ \\
& \left(233+146 x+50 x^2\right)t^7+ \left(450+498 x+357 x^2+125 x^3\right)t^8+\\
& \left(794+1299
x+1429 x^2+990 x^3+350 x^4\right)t^9+\\
& \left(1307+2869 x+4263 x^2+4353 x^3+2954 x^4+1050 x^5\right)t^{10}+ \cdots .
\end{align*}

\vspace{-0.5cm}

\begin{align*}
&Q_{132}^{(2,2,0,2)}(t,x) = 1+t+2 t^2+5 t^3+14 t^4+42 t^5+132 t^6+ (348+81 x)t^7+\ \ \\
& \left(811+457 x+162 x^2\right)t^8+ \left(1747+1625 x+1085 x^2+405 x^3\right)t^9+\\
&
\left(3587+4663 x+4443 x^2+2969 x^3+1134 x^4\right)t^{10}+\\
& \left(7167+11864 x+14360 x^2+13201 x^3+8792 x^4+3402 x^5\right)t^{11}+ \cdots .
\end{align*}

\vspace{-0.5cm}

\begin{align*}
&Q_{132}^{(3,2,0,2)}(t,x) =  1+t+2 t^2+5 t^3+14 t^4+42 t^5+132 t^6+429 t^7+
\ \ \ \ \ \ \ \ \ \ \ \\
& (1234+196 x) t^8 +2  \left(1657+578 x+196 x^2\right)t^9+\\
&  \left(8643+4497 x+2676 x^2+980
x^3\right)t^{10}+\\
& \left(22345+14839 x+11622 x^2+7236 x^3+2744 x^4\right)t^{11}+\cdots .
\end{align*}

\begin{problem} In all the cases above, it seems that for 
$n \geq \ell +k +m+1$, the highest power of $x$ in
$Q_{n,132}^{(\ell,k,0,m)}(x)$ is $x^{n-k-\ell}$ which appears 
with a coefficient of $a_{\ell,k,m}C_{n-\ell-k-m}$ for some 
constant $a_{\ell,k,m}$.  Prove that this is the case and 
find a formula for $a_{\ell,k,m}$.
\end{problem}

\section{$Q_{n,132}^{(a,b,c,d)}(x) = Q_{n,132}^{(a,d,c,b)}(x)$ 
where $a,b,c,d \geq 1$}

By Lemma \ref{sym},  we only need to consider the case of $Q_{n,132}^{(a,b,c,d)}(x)$. Suppose that $a,b,c,d \geq 1$ and $n \geq b+d$.  
It is clear that $n$ can never match 
the pattern $\MMP(a,b,c,d)$ for $a,b,c,d \geq 1$ in any 
$\sg \in S_n(132)$. If $\sg = \sg_1 \ldots \sg_n \in S_n(132)$  
and $\sg_i =n$, then we have three cases, depending 
on the value of $i$. \\
\ \\
{\bf Case 1.} $i < b$. It is easy to see that as we sum 
over all the permutations $\sg$ in $S_n^{(i)}(132)$, our choices 
for the structure for $A_i(\sg)$ will contribute a factor 
of $C_{i-1}$ to $Q_{n,132}^{(a,b,c,d)}(x)$ since 
the elements in $A_i(\sg)$ do not have enough elements 
to the left to match $\MMP(a,b,c,d)$ in $\sg$. 
Similarly, our choices 
for the structure for $B_i(\sg)$ will contribute a factor 
of $Q_{n-i,132}^{(a,b-i,c,d)}(x)$ to $Q_{n,132}^{(a,b,c,d)}(x)$ 
since $\sg_1 \ldots \sg_i$ will automatically be 
in the second quadrant relative to the coordinate system 
with the origin at $(s,\sg_s)$ for any $s > i$.  Thus, the 
permutations in Case 1 will contribute 
$$\sum_{i=1}^{b-1} C_{i-1}
Q_{n-i,132}^{(a,b-i,c,d)}(x)$$ 
to $Q_{n,132}^{(a,b,c,d)}(x)$.\\
\ \\
{\bf Case 2.} $b \leq i < n-d$. It is easy to see that as we sum 
over all the permutations $\sg$ in $S_n^{(i)}(132)$, our choices 
for the structure for $A_i(\sg)$ will contribute a factor 
of $Q_{i-1,132}^{(a-1,b,c,0)}(x)$ to $Q_{n,132}^{(a,b,c,d)}(x)$ since 
the elements in $B_i(\sg)$ will all be in the fourth quadrant  
and $\sg_i =n$ is in the first quadrant
relative to a coordinate system centered at $(r,\sg_r)$ for 
$r \leq i$  in this case. 
Similarly, our choices 
for the structure for $B_i(\sg)$ will contribute a factor 
of $Q_{n-i,132}^{(a,0,c,d)}(x)$ to $Q_{n,132}^{(a,b,c,d)}(x)$ 
since $\sg_1 \ldots \sg_i$ will automatically be 
in the second quadrant relative to the coordinate system 
with the origin at $(s,\sg_s)$ for any $s > i$.  Thus, the 
permutations in Case 2 will contribute 
$$\sum_{i=b}^{n-d} Q_{i-1,132}^{(a-1,b,c,0)}(x) 
Q_{n-i,132}^{(a,0,c,d)}(x)$$ 
to $Q_{n,132}^{(a,b,c,d)}(x)$.\\
\ \\
{\bf Case 3.} $i \geq n-d +1$. It is easy to see that as we sum 
over all the permutations $\sg$ in $S_n^{(i)}(132)$, our choices 
for the structure for $A_i(\sg)$ will contribute a factor 
of $Q_{i-1,132}^{(a-1,b,c,d-(n-i))}(x)$ 
to $Q_{n,132}^{(a,b,c,d)}(x)$ since $\sg_i =n$ 
will be in the first quadrant and 
the elements in $B_i(\sg)$ will all be in the fourth quadrant 
relative to a coordinate system centered at $(r,\sg_r)$ for 
$r \leq i$ in this case. 
Similarly, our choices 
for the structure for $B_i(\sg)$ will contribute a factor 
of $C_{n-i}$ to $Q_{n,132}^{(a,b,c,d)}(x)$ 
since $\sg_j$, where $j >i$, does not have enough elements 
to its right to match $\MMP(a,b,c,d)$ in $\sg$.  Thus, the 
permutations in Case 3 will contribute 
$$\sum_{i=n-d+1}^{n} Q_{i-1,132}^{(a-1,b,c,d-(n-i))}(x) 
C_{n-i}$$ 
to $Q_{n,132}^{(a,b,c,d)}(x)$. Thus, 
we have the following. For $ n \geq a+b+c+d+1 $,  
\begin{eqnarray}\label{sbcdrec} 
Q_{n,132}^{(a,b,c,d)}(x) &=& \sum_{i=1}^{b-1} C_{i-1} 
Q_{n-i,132}^{(a,b-i,c,d)}(x) + 
\sum_{i=b}^{n-d} Q_{i-1,132}^{(a-1,b,c,0)}(x) 
Q_{n-i,132}^{(a,0,c.d)}(x) +\nonumber \\
&& \sum_{i=n-d+1}^{n} Q_{i-1,132}^{(a-1,b,c,d-(n-i))}(x) C_{n-i}.
\end{eqnarray}
Multiplying (\ref{sbcdrec}) by $t^n$ and summing, we obtain 
the following theorem. 
\begin{theorem} For all $a,b,c,d \geq 1$, 
\begin{eqnarray*}\label{eq:abcd0}
Q_{132}^{(a,b,c,d)}(t,x) &=& \sum_{p=0}^{b+d-1} C_p t^p + \nonumber \\
&&t\sum_{i=0}^{b-2} C_i t^i \left(  Q_{132}^{(a,b-1-i,c,d)}(t,x) -
\sum_{r=0}^{b-i+d-2} C_rt^r \right) + \nonumber \\
&&t \left( Q_{132}^{(a-1,b,c,0)}(t,x) - \sum_{i=0}^{b-2} C_i t^i\right) 
\left( Q_{132}^{(a,0,c,d)}(t,x) - \sum_{j=0}^{d-1} C_j t^j \right) + 
\nonumber \\
&& t\sum_{j=0}^{d-1} C_j t^j \left( Q_{132}^{(a-1,b,c,d-j)}(t,x) - 
\sum_{s=0}^{b+d-j-2} C_st^s\right).
\end{eqnarray*}
\end{theorem}

Thus, for example,
\begin{eqnarray*}\label{eq:1111}
Q_{132}^{(1,1,1,1)}(t,x) &=& 1+t +
tQ_{132}^{(0,1,1,0)}(t,x)\left(Q_{132}^{(1,0,1,1)}(t,x) -1\right)+ \nonumber\\
&&t(Q_{132}^{(0,1,1,1)}(t,x)-1).
\end{eqnarray*}
and, for $k \geq 2$, 
\begin{eqnarray*}\label{eq:k111}
Q_{132}^{(k,1,1,1)}(t,x) &=& 1+t +
tQ_{132}^{(k-1,1,1,0)}(t,x)\left(Q_{132}^{(k,0,1,1)}(t,x) -1\right)+ \nonumber\\
&&t(Q_{132}^{(k-1,1,1,1)}(t,x)-1).
\end{eqnarray*}

\vspace{-0.5cm}

\begin{eqnarray*}
&&Q_{132}^{(1,1,1,1)}(t,x) =  1+t+2 t^2+5 t^3+14 t^4+ (38+4 x)t^5+ 
\left(99+29 x+4 x^2\right)t^6+\\
&& \left(249+135 x+41 x^2+4 x^3\right)t^7+ \left(609+510 x+250 x^2+57
x^3+4 x^4\right)t^8+\\
&& \left(1457+1701 x+1177 x^2+446 x^3+77 x^4+4 x^5\right)t^9+\\
&& \left(3425+5220 x+4723 x^2+2564 x^3+759 x^4+101 x^5+4 x^6\right)t^{10}+
\cdots, 
\end{eqnarray*}

\begin{eqnarray*}
&&Q_{132}^{(2,1,1,1)}(t,x) = 
1+t+2 t^2+5 t^3+14 t^4+42 t^5+ (123+9x)t^6+ \ \ \ \ \ \ \ \ \ \ \ \  \\
&&\left(350+70 x+9 x^2\right)t^7+ \left(974+350 x+97 x^2+9 x^3\right)t^8+ \\
&&\left(2667+1433
x+620 x^2+133 x^3+9 x^4\right)t^9+ \\
&&\left(7218+5235 x+3079 x^2+1077 x^3+178 x^4+9 x^5\right)t^{10} + \cdots, 
\end{eqnarray*}
and 
\begin{eqnarray*}
&&Q_{132}^{(3,1,1,1)}(t,x) = 1+t+2 t^2+5 t^3+14 t^4+42 t^5+132 t^6+ 
(413+16 x)t^7+ \ \ \ \   \\
&&\left(1277+137 x+16 x^2\right)t^8+ 
\left(3909+752 x+185 x^2+16 x^3\right)t^9+ \\
&&\left(11881+3383 x+1267 x^2+249 x^3+16 x^4\right)t^{10} + \cdots .
\end{eqnarray*}

It is easy to explain the coefficient to the highest power that appears in 
$Q_{n,132}^{(k,1,1,1)}(x)$ for $k \geq 1$.  That is, the maximum number of 
matches of $\MMP(1,1,1,1)$ for $\sg \in S_n(132)$ is when 
$\sg$ is of the form $x~\alpha~\beta$ where 
$x \in \{n-k, \ldots, n\}$, $\beta$ is a shuffle of 1 with 
the sequence $(n-k) (n-k+1) \ldots n$ with $x$ removed,  and 
$\alpha = 23 \ldots (n-k-1)$. Note that we have $k+1$ choices for 
$x$ and, once we chosen $x$, we have $k+1$ choices for $\beta$. 
Thus, the highest power 
of $x$ that occurs in $Q_{n,132}^{(k,1,1,1)}(x)$ is $x^{n-k-3}$ which occurs with a coefficient of $(k+1)^2$ for $n \geq k+4$. 
 
We also have
\begin{eqnarray*} 
Q_{132}^{(0,1,1,0)}(t,0) &=& \frac{1-t}{1-2t}, \\
Q_{132}^{(1,0,1,1)}(t,0) &=& 1+t\left(\frac{1-t}{1-2t}\right)^2, \ \mbox{and}\\
Q_{132}^{(0,1,1,1)}(t,0) &=& \frac{1-4t+5t^2-t^3}{(1-2t)^2(1-t)},
\end{eqnarray*}
to compute that 
\begin{equation*}
Q_{132}^{(1,1,1,1)}(t,0) = \frac{1-6t+13t^2-11t^3+3t^4-2t^5+t^6}{(1-t)(1-2t)^3}.
\end{equation*}

Note that $Q_{132}^{(1,1,1,1)}(t,0)$ is the generating function of 
the permutations that avoid the patterns from the set
$\{132,52314,52341,42315,42351\}$.

Finally, we can also determine the second highest coefficient 
of $x$ in $Q_{n,132}^{(1,1,1,1)}(x)$. 

\begin{theorem} For all $n \geq 6$, 
 $$Q_{n,132}^{(1,1,1,1)}(x)|_{x^{n-5}} = 17+4\binom{n-3}{3}.$$ 
\end{theorem}
\begin{proof}
The recursion of $Q_{n,132}^{(1,1,1,1)}(x)$ is 
\begin{equation*}
Q_{n,132}^{(1,1,1,1)}(x) = Q_{n-1,132}^{(0,1,1,1)}(x) +
\sum_{i=1}^{n-1} Q_{i-1,132}^{(0,1,1,0)}(x) Q_{n-i,132}^{(1,0,1,1)}(x).
\end{equation*}
For $n \geq 3$, the highest power of $x$ which occurs in 
$Q_{n,132}^{(0,1,1,0)}(x)$ is $x^{n-2}$ and for $n \geq 4$, 
the highest power of $x$ that occurs in $Q_{n,132}^{(1,0,1,1)}(x)$ is 
$x^{n-3}$. It follows that for $i =2,\ldots, n-3$, the 
highest power of $x$ that occurs in 
$Q_{i-1,132}^{(0,1,1,0)}(x) Q_{n-i,132}^{(1,0,1,1)}(x)$ is $x^{n-6}$. 
It follows that 
\begin{eqnarray*}
Q_{n,132}^{(1,1,1,1)}(x)|_{x^{n-5}} &=&
Q_{n-1,132}^{(1,0,1,1)}(x)|_{x^{n-5}} +
Q_{n-2,132}^{(1,0,1,1)}(x)|_{x^{n-5}} + 
2Q_{n-3,132}^{(0,1,1,0)}(x)|_{x^{n-5}}+\\
&&Q_{n-2,132}^{(0,1,1,0)}(x)|_{x^{n-5}}+
Q_{n-1,132}^{(0,1,1,1)}(x)|_{x^{n-5}}.
\end{eqnarray*}
But for $n \geq 6$, we have proved that 
\begin{eqnarray*}
Q_{n-1,132}^{(1,0,1,1)}(x)|_{x^{n-5}} &=& 6+2\binom{n-3}{2},\\
Q_{n-2,132}^{(1,0,1,1)}(x)|_{x^{n-5}} &=& 2, \\ 
2Q_{n-3,132}^{(0,1,1,0)}(x)|_{x^{n-5}} &=& 2C_1 =2, \\
Q_{n-2,132}^{(0,1,1,0)}(x)|_{x^{n-5}} &=& 2 +\binom{n-3}{2}, \ \mbox{and}\\
Q_{n-1,132}^{(0,1,1,1)}(x)|_{x^{n-5}} &=& 5 + \binom{n-3}{2}. 
\end{eqnarray*}
Thus, $Q_{n,132}^{(1,1,1,1)}(x)|_{x^{n-5}} = 17 +4 \binom{n-3}{2}$.
\end{proof}

\end{document}